\documentclass[11pt, leqno, a4paper,twoside]{amsart} 
\usepackage{lipsum}
\usepackage{amsfonts}
\usepackage{graphicx}
\usepackage{epstopdf}
\usepackage{algorithmic}
\usepackage{calligra}
\usepackage{amsfonts,amsmath,amsthm,amssymb}
\usepackage{mathtools}
\usepackage{hyperref}
\usepackage{autonum}
\usepackage{hhline}
\usepackage{array}
\usepackage{diagbox}
\usepackage{tcolorbox}
\usepackage{mdframed}
\usepackage{multicol}
\usepackage{graphicx}
\usepackage{subcaption}
\usepackage{moreverb}
\usepackage{bbm}
\usepackage[margin=1.4in]{geometry}
\usepackage{todonotes}
\usepackage{scalerel,amssymb}
\allowdisplaybreaks
\usepackage{mathrsfs}  
\usepackage{lineno}
\usepackage{todonotes}
\usepackage{tikz}
\usepackage{pgfplots}
\usetikzlibrary{arrows.meta}
\usepackage[numbers,sort&compress]{natbib}

\def\ooal{\alpha^{-1}}
\def\ulal{\underline{\alpha}}
\def\olal{\overline{\alpha}}
\def\op{\overline{p}}
\def\opsi{\overline{\psi}}
\def\oq{\overline{q}}
\newcommand{\Om}{\Omega}
\newcommand{\D}{\Delta}
\newcommand{\pn}{p^n}
\DeclareMathOperator*{\esssup}{ess\,sup}

\newcommand{\pt}{p_t}
\newcommand{\ptn}{p_t^n}
\newcommand{\ptt}{p_{tt}}
\newcommand{\pttn}{p_{tt}^n}

\newcommand{\ddt}{\frac{\textup{d}}{\textup{d}t}}
\newcommand{\dt}{\, \textup{d} t}
\newcommand{\ds}{\, \textup{d} s }
\newcommand{\dx}{\, \textup{d} x}
\newcommand{\dxs}{\, \textup{d}x\textup{d}s}

\newcommand{\intTO}{\int_0^t \int_{\Omega}}

\newcommand{\intO}{\int_{\Omega}}
\newcommand{\nLtwo}[1]{\|#1\|_{L^2}}
\newcommand{\nLthree}[1]{\|#1\|_{L^3}}
\newcommand{\nLfour}[1]{\|#1\|_{L^4}}
\newcommand{\nLsix}[1]{\|#1\|_{L^6}}
\newcommand{\nLinf}[1]{\|#1\|_{L^\infty}}

\newcommand{\nLtwoLtwo}[1]{\|#1\|_{L^2 (L^2)}}
\newcommand{\nLtwoLfour}[1]{\|#1\|_{L^2 (L^4)}}
\newcommand{\nLtwoLsix}[1]{\|#1\|_{L^2 (L^6)}}
\newcommand{\nLtwoLinf}[1]{\|#1\|_{L^2 (L^\infty)}}
\newcommand{\nLinfLtwo}[1]{\|#1\|_{L^\infty (L^2)}}
\newcommand{\nLinfLthree}[1]{\|#1\|_{L^\infty (L^3)}}
\newcommand{\nLinfLinf}[1]{\|#1\|_{L^\infty (L^\infty)}}

\newcommand{\nLinfHtwo}[1]{\|#1\|_{L^\infty (H^2)}}
\newcommand{\nLinfHthree}[1]{\|#1\|_{L^\infty (H^3)}}

\newcommand{\vardbtilde}[1]{\tilde{\raisebox{0pt}[0.85\height]{$\tilde{#1}$}}} 
\newcommand{\prodLtwo}[2]{(#1, #2)_{L^2}}
\newcommand{\R}{\mathbb{R}} 
\newcommand{\N}{\mathbb{N}} 
\newcommand{\Ltwo}{L^2(\Omega)}
\newcommand{\Hone}{H^1(\Omega)}
\newcommand{\Hthree}{H^3(\Omega)}
\newcommand{\Honetwo}{{H_\diamondsuit^2(\Omega)}}
\newcommand{\Honethree}{{H_\diamondsuit^3(\Omega)}}
\newcommand{\Honefour}{{H_\diamondsuit^4(\Omega)}}
\newcommand{\spaceW}{X^{\textup{W}}}
\newcommand{\spaceK}{X^{\textup{K}}}

\newcommand{\LinfLinf}{L^\infty(0,T; L^\infty(\Omega))}
\newcommand{\CHone}{C_{H^1, L^4}}
\newcommand{\CHonesix}{C_{H^1, L^6}}
\newcommand{\CHtwo}{C_{H^2, L^\infty}}
\newcommand{\CHtwoeight}{C_{H^2, L^8}}
\newcommand{\CPF}{C_{\textup{PF}}}

\newtheorem{theorem}{Theorem}

\newtheorem{proposition}{Proposition}
\newtheorem*{assumption*}{Assumptions}

\newtheorem{remark}{Remark}
\numberwithin{lemma}{section}
\numberwithin{proposition}{section}
\numberwithin{theorem}{section}
\numberwithin{equation}{section}
\makeatletter
\newcommand{\leqnomode}{\tagsleft@true}
\newcommand{\reqnomode}{\tagsleft@false}
\makeatother
\newcommand{\TK}{\mathcal{T}_{\textup{K}}} 
\newcommand{\MK}{M_{\textup{K}}} 
\newcommand{\RK}{R_{\textup{K}}} 

\definecolor{grey}{rgb}{0.5,0.5,0.5}
\newcommand{\abssig}{|\sigma|}

\title[Parabolic approximation of quasilinear wave equations]{Parabolic approximation of quasilinear wave equations with applications in nonlinear acoustics}


\author[B. Kaltenbacher and V. Nikoli\'c]{\bfseries Barbara Kaltenbacher$^1$ and Vanja Nikoli\'c$^2$}

\address{  
	Department of Mathematics\\  
	Alpen-Adria-Universit\"at Klagenfurt 
	\\ 9020 Klagenfurt, Austria}
\email{barbara.kaltenbacher@aau.at}

\address{ 
	Department of Mathematics \\ 
	Radboud University   \\ 
	Heyendaalseweg 135,
	6525 AJ Nijmegen, The Netherlands}
\email{vanja.nikolic@ru.nl} 

\begin{document}
	\maketitle     
	\vspace*{-4mm}      
	\begin{center}
		{\footnotesize
			$^1$Department of Mathematics, Alpen-Adria-Universit\"at Klagenfurt, Austria \\
			$^2$Department of Mathematics, Radboud University, The Netherlands
		}
	\end{center}
	\vspace*{4mm}
\begin{abstract}
This work deals with the convergence analysis of parabolic perturbations to quasilinear wave equations on smooth bounded domains. In particular, we consider wave equations with nonlinearities of quadratic type, which cover the two classical models of nonlinear acoustics, the Westervelt and Kuznetsov equations. By employing a high-order energy analysis, we obtain convergence of their solutions to the corresponding inviscid equations' solutions in the standard energy norm with a linear rate, assuming small data and a sufficiently short time. The smallness of initial data can, however, be imposed in a lower-order norm than the one needed in the energy analysis. It arises only from ensuring the non-degeneracy of the studied wave equations. In addition, we address the open questions of local well-posedness of the two classical models on bounded domains with a non-negative, possibly vanishing sound diffusivity.
\end{abstract}
	
	\section{Introduction}
	\indent  This work is concerned with the convergence analysis of solutions to nonlinear wave equations in the form of
	\begin{equation} \label{Westervelt_like}
	\begin{aligned}
	& \alpha (p)p_{tt}-c^{2}\Delta p-b \Delta p_t=f(p_t)
	\end{aligned}
	\end{equation}
	as well as 
	\begin{equation}   \label{Kuznetsov_like}
	\begin{aligned}
	& \alpha(\psi_t)\psi_{tt}-c^{2}\Delta \psi-b \Delta \psi_{t}=f(\nabla \psi, \nabla \psi_{t}),
	\end{aligned}
	\end{equation}
	as the damping parameter $b$ tends to zero. These equations can be understood as perturbations of quasi-linear hyperbolic models with $b=0$ into parabolic-like equations with $b>0$. The error analysis of such perturbations is of interest, for example, in numerical simulations, where artificial viscosity is needed to resolve the nonlinear steepening of quasilinear waves by taming spurious oscillations at the wave peaks; see, for example,~\cite{hoffelner2001finite, tsuchiya1992simulation, soneson2019parabolic} for more details on this issue. {The present work also addresses the well-posedness of \eqref{Westervelt_like} and \eqref{Kuznetsov_like} on bounded domains which holds uniformly for $b\geq 0$, under the assumption of having smooth data that are small in possibly low-order topology}.\\
	\indent We are particularly motivated in our research by applications of high-intensity ultrasonic waves, whose behavior is inherently nonlinear and modeled by nonlinear wave equations of this type. These applications range from non-invasive treatments of different medical disorders to non-destructive material testing and urge the need to accurately model and analyze ultrasonic waves in different settings. The thermoviscous dissipation of fluids is known to play an important role in the propagation of sound waves. In classical acoustic models, such as the Westervelt and Kuznetsov equations, these effects are encompassed within the medium parameter $b$, often referred to as the sound diffusivity; see~\cite{lighthill2001waves, westervelt1963parametric, kuznetsov1971equations, kaltenbacher2007numerical}. Thus, the present work can also be understood as an investigation into how the behavior of solutions to nonlinear acoustic equations that model ultrasonic propagation changes with the sound diffusivity. We refer to Section~\ref{Sec:ProblemSetting} below for further details on modeling in nonlinear acoustics. Formally, setting $b =0$ in \eqref{Westervelt_like} and \eqref{Kuznetsov_like} corresponds to lossless nonlinear sound propagation through inviscid fluids or gases; see, e.g., \cite{ChristovChristovJordan2007} and the references therein.  Corresponding to the applications in nonlinear acoustics, we here focus on quadratic nonlinearities \[f(p_t)=kp_t^2, \quad \alpha(p)=1-kp\] in equation \eqref{Westervelt_like} and \[f(\nabla \psi,\nabla \psi_{t})=\nabla \psi\cdot\nabla \psi_{t},\quad \alpha(\psi_t)=1-\kappa\psi_t\] in equation \eqref{Kuznetsov_like}, but wish to emphasize that our analysis carries over to more general nonlinearities under appropriate growth conditions on $f$.
	\\
	\indent We structure the paper as follows. In Section~\ref{Sec:ProblemSetting}, we set up the problem, provide an overview of related results, and discuss the main contributions of the present work. Section~\ref{Sec:LinWest} is dedicated to the well-posedness of a linear version of equation \eqref{Westervelt_like}. In particular, the main aim of that section is to arrive at a higher-order energy bound that is uniform with respect to the parameter $b$. In Section~\ref{Sec:West}, we employ Banach's fixed-point theorem on a suitably defined mapping to analyze \eqref{Westervelt_like} when the data are sufficiently smooth and small.  Section~\ref{Sec:LimitWest} contains the convergence analysis for this equation in the standard energy norm as well as a higher-order norm as $b \rightarrow 0^+$ for small data. In Section~\ref{Sec:Kuznetsov}, we extend our energy analysis to \eqref{Kuznetsov_like}. Section~\ref{Sec:LimitKuzn} is then devoted to the convergence rate of solutions to this equation as $b \rightarrow 0^+$. We conclude the paper with a discussion on the obtained results and an outlook on future work. 
	\section{Modeling and problem setting} \label{Sec:ProblemSetting}
	Propagation of nonlinear sound waves through thermoviscous fluids is often modeled by the Kuznetsov equation
	\begin{equation}
	\psi_{tt}-c^{2}\Delta \psi-b \Delta \psi_{t}=\left( \frac{1}{c^{2}}\frac{B}{2A}\psi_{t}^{2}+|\nabla \psi|^{2}\right)_t,
	\end{equation}
	given in terms of the acoustic velocity potential $\psi=\psi(x,t)$; see~\cite{kuznetsov1971equations, kaltenbacher2007numerical} for its derivation. The coefficient $c>0$ stands for the speed of sound in the fluid, whereas the ratio $B/A$ arises from the equation of state, which relates variations of the acoustic pressure to the ones of medium density. The dissipative mechanisms due to the viscosity and thermal conductivity of the medium are accounted for via the coefficient $b$, known as the sound diffusivity; cf.~\cite{lighthill2001waves}. \\ 
	\indent In many instances, cumulative nonlinear effects in sound propagation dominate the local ones. This the case, for example, when the traveling distance becomes much greater than a wavelength; see~\cite{hamilton1998nonlinear} for a detailed discussion on this topic. The Westervelt equation is then a suitable second-order approximation of the full set of governing equations. In terms of the acoustic potential, it is given by
	\begin{equation} \label{West_potential}
	\psi_{tt}-c^{2}\Delta \psi-b \Delta \psi_{t}=\frac{1}{c^{2}}\left(\frac{B}{2A}+1\right)(\psi_t^2)_t.
	\end{equation}
	However, more commonly it is expressed via the acoustic pressure $p= \varrho \psi_t$ as
	\begin{equation}
	p_{tt}-c^{2}\Delta p-b \Delta p=\frac{1}{\varrho c^{2}}\left(\frac{B}{2A}+1\right)(p^2)_{tt},
	\end{equation}
	where $\varrho$ denotes the medium density; see~\cite{westervelt1963parametric} for its original derivation in the lossless form. The case $b=0$ corresponds to sound propagation through lossless media. \\
	\indent Motivated by these classical models in nonlinear acoustics, we here study the behavior of solutions to nonlinear wave equations in the form of
	\begin{equation} \label{Kuznt}
	\psi_{tt}-c^{2}\Delta \psi-b \Delta \psi_{t}= \frac12 (\kappa \psi_t^2+ \sigma |\nabla \psi|^2)_t, 
	\end{equation}
	and
	\begin{equation} \label{West}
	p_{tt}-c^{2}\Delta p-b \Delta p= \frac12 (k p^2)_{tt}
	\end{equation}
	as the positive parameter $b$ vanishes. Particular choices of $k$, $\kappa$, $\sigma \in \R$ will lead to one of the above classical models of nonlinear acoustics.  Henceforth, we will refer to \eqref{Kuznt} and \eqref{West} as the Kuznetsov and Westervelt equation, respectively. 
	\subsection{{Related literature}} {
		The global well-posedness of equations \eqref{Kuznt} and \eqref{West}  for
		fixed positive $b>0$ in the case of homogeneous Dirichlet boundary conditions
		and small, smooth initial data has been studied in~\cite{mizohata1993global, kawashima1992global} and \cite{kaltenbacher2009global, Wilke}, respectively. In the
		inviscid case with $b=0$, smooth solutions can be expected to exist only locally in time. This
		observation has been made rigorous in~\cite{dorfler2016local} for the Westervelt equation. The Kuznetsov and Westervelt equations with $b>0$ have also been analyzed in~\cite{bongarti2021vanishing} and \cite{KaltenbacherNikolic}, respectively, in the limit of third-order acoustic models in time for a vanishing thermal relaxation parameter. \\
		\indent In $\R^n$, local well-posedness of the inviscid Westervelt and Kuznetsov equations in potential form follows as a particular case of the analysis performed in~\cite{hughes1977well}. The Cauchy problem for the Kuznetsov equation and its inviscid version has been further studied in~\cite{dekkers2017cauchy} for small enough data. In~\cite{tani2017mathematical}, convergence of solutions in $\R^3$ in the vanishing sound diffusivity limit for an enhanced version of the Kuznetsov equation that involves additionally a $\psi_t\D \psi$ term on the right-hand side has been studied. Compared to our results in Theorem~\ref{Thm:Kuzn_WeakLimit} below, convergence is obtained in $d=3$ in a higher-order norm for small initial data in $H^{2+l}(\R^3) \times H^{1+l}(\R^3)$, where $l >1$, $l-1/2 \notin 2\N$; see~\cite[Theorem 4]{tani2017mathematical}. }
	\subsection{Main contributions}
	{The novel contributions of our work pertain to the convergence analysis of the solutions to \eqref{Kuznt} and \eqref{West} with homogeneous Dirichlet conditions as $b \rightarrow 0^+$ and where data are allowed to be large in higher-order topologies}. {We prove} convergence of solutions in the standard energy norm with a linear rate to the solutions of the corresponding inviscid equations:
	\begin{equation}
	\begin{aligned}
	&\|p^{(b)}-p^{(0)}\|_{\textup{E}}\lesssim b, \\ 
	& \|\psi^{(b)}-\psi^{(0)}\|_{\textup{E}}\lesssim b, \\ 
	\end{aligned}
	\end{equation}
	as $b \rightarrow 0^+$. These results are contained in Theorems~\ref{Thm:West_WeakLimit} and \ref{Thm:Kuzn_WeakLimit}  below. Provided that the final time is short enough, the smallness conditions on the initial data can be imposed in a lower-order norm than the one needed in the energy analysis. We expect these estimates to be of practical relevance in numerical simulations of quasilinear waves, where artificial dissipation is commonly introduced to combat the spurious oscillations that arise as the wavefront steepens, also known as Gibbs oscillations; cf.~\cite{walsh2007finite, kaltenbacher2007numerical}. {As a by-product of our analysis, we also address the questions of local well-posedness of \eqref{Kuznt} and \eqref{West} on bounded domains for $b \geq 0$, which, to the best of our knowledge, have not been answered before for $b=0$ in the Kuznetsov case \eqref{Kuznt}. Further, compared to the existing result on analysis of the inviscid Westervelt equation~\cite{dorfler2016local}, we allow for arbitrarily large data in $H^2(\Omega) \times H^1(\Omega)$, as relevant in practical ultrasonic applications.} We point the readers to Theorems~\ref{Thm:West_Wellposedness} and~\ref{Thm:WellpKuzn} below for the respective results.  \\
	\indent The convergence analysis of quasi-linear wave equations appears to be fundamentally different from the linear analysis performed in~\cite{showalter1976regularization} in that it requires stronger regularity of initial data to obtain bounds that are independent of $b$. Namely, \[(p(0),p_t(0))\in H^3(\Omega)\times H^2(\Omega)\] in the Westervelt, and \[(\psi(0),\psi_t(0))\in H^4(\Omega)\times H^3(\Omega)\] in the Kuznetsov case, as compared to \[(p(0),p_t(0))\in \Hone\times\Ltwo\] in the linear case. This higher regularity assumption allows us to obtain an $O(b)$ convergence rate compared to~\cite[Proposition 7]{showalter1976regularization}, where an $O(\sqrt{b})$ convergence rate is established with $H^1(0,T;H^1(\Omega))$ regularity of the solution (that is, $H^2(\Omega)\times H^1(\Omega)$ -- actually a bit less -- regularity of the initial data $(p(0),p_t(0))$). Note that we also provide an $O(\sqrt{b})$ rate under the initial regularity $(p(0),p_t(0))\in H^3(\Omega)\times H^2(\Omega)$ for the Westervelt equation in a stronger norm than the standard energy one; cf. Remark~\ref{rem:rate}.
	\subsection{Notation and useful inequalities}
	For simplicity of notation, we often omit the time interval and the spatial domain when writing norms; i.e., $\|\cdot\|_{L^p (L^q)}$ denotes the norm on $L^p(0,T;L^q(\Omega))$. While $H^s(\Omega)$ denotes the usual Sobolev spaces, we will denote by 
	\begin{equation} \label{sobolev_withtraces}
	\begin{aligned}
	\Honetwo=&\,H_0^1(\Omega)\cap H^2(\Omega),\\
	\Honethree=&\, \left\{u\in H^3(\Omega)\,:\, \mbox{tr}_{\partial\Omega} u = 0, \  \mbox{tr}_{\partial\Omega} \D u = 0\right\},\\ 
	\Honefour=&\,\left\{u\in H^4(\Omega)\,:\, \mbox{tr}_{\partial\Omega} u = 0, \  \mbox{tr}_{\partial\Omega} \D u = 0 \right\},
	\end{aligned}
	\end{equation}
	spaces of functions $H_0^2(\Omega)\subset \Honetwo\subset H^2(\Omega)$, $H_0^3(\Omega)\subset \Honethree\subset H^3(\Omega)$, and $H_0^4(\Omega)\subset \Honefour\subset H^4(\Omega)$ satisfying certain boundary conditions. By $\hookrightarrow\hookrightarrow$ we denote compact embeddings, by $\hookrightarrow$ continuous ones.
	\\
	\indent Throughout the paper, we often employ 
	boundedness of the operator $(-\Delta)^{-1}:L^2(\Omega)\to \Honetwo$ and the Poincar\'{e}--Friedrichs inequality: 
	\begin{equation}\label{ellreg-PF}
	\begin{aligned}
	&\|v\|_{H^2(\Omega)}\leq C_{(-\Delta)^{-1}} \|-\Delta v\|_{L^2(\Omega)}, \quad &&v \in \Honetwo; \\
	&\|v\|_{H^1(\Omega)}\leq \CPF \|\nabla v\|_{L^2(\Omega)}, \quad &&v \in H^1_0(\Omega).
	\end{aligned}
	\end{equation}
	Furthermore, we rely on the continuous embeddings 
	$H^1(\Omega)\hookrightarrow L^6(\Omega)$ and $H^2(\Omega) \hookrightarrow L^\infty(\Omega)$:
	\begin{equation}\label{embeddigs}
	\begin{aligned}
	&\|v\|_{L^6(\Omega)}\leq \CHonesix \|\nabla v\|_{L^2(\Omega)}, \quad && v \in H^1_0(\Omega)\\
	&\|v\|_{L^\infty(\Omega)}\leq C_{H^2, L^\infty} \|\D v\|_{L^2(\Omega)}, \quad && v \in \Honetwo.
	\end{aligned}
	\end{equation}
	\indent We often use $x \lesssim y$ to denote $x \leq C y$, where the constant $C>0$ does not depend on the sound diffusivity $b$, but may depend on other medium parameters and the final time $T$.
\section{Energy analysis of the corresponding linear problem for the Westervelt equation}  \label{Sec:LinWest}
\indent In this section, we study the linear version of the initial-boundary value problem for the Westervelt equation:
\begin{equation} \label{ibvp_linWest}
\left\{
\begin{aligned}
\alpha \ptt - b \D p_t - c^2\D p - k q_t p_t =&\, 0  &&\quad\text{ in }\Omega\times(0,T), \\[1mm]
p=&\,0  &&\quad\text{ on } \partial \Omega\times(0,T),\\[1mm]
(p, p_t)=&\,(p_0, p_1)  &&\quad\mbox{ in }\Omega\times \{0\},
\end{aligned} \right.
\end{equation}
where $q$ is a given function, $\alpha= 1-kq$, and $k \in \R $. Our aim is to obtain a unique solution that is uniformly bounded with respect to $b$. We analyze the problem under the non-degeneracy assumption on the coefficient $\alpha=\alpha(x,t)$. In other words, we assume that there exist $\olal$, $\ulal>0$, such that
\begin{equation} \label{non-degeneracy_assumption}
\ulal \leq \alpha(x, t)\leq \olal \ \mbox{ on }\Omega \ \ \text{  a.e. in } \Om \times (0,T).
\end{equation}
The coefficient $\alpha$ is required to have the following regularity:
\begin{equation}\label{eq:alphagammaf_reg_Wes}
\begin{aligned} 
&\alpha \in \spaceW_{\alpha}=L^\infty(0,T; H^3(\Om)) \cap W^{1, \infty}(0,T; H^2(\Om)).
\end{aligned}
\end{equation}
Note that then
\begin{equation}\label{eq:regq}
q_t=-\frac{1}{k}\alpha_t \in L^{\infty}(0,T; H^2(\Om).
\end{equation}
The space $\spaceW_{\alpha}$ is equipped with the norm
\begin{equation}
\begin{aligned}
\|\alpha\|_{\spaceW_{\alpha}}= \esssup_{t \in (0,T)}\|\alpha(t)\|_{H^3}+\esssup_{t \in (0,T)} \|\alpha_t(t)\|_{H^2}.
\end{aligned}
\end{equation}
Moreover, we impose the following regularity on the initial data: 
\begin{equation}\label{IC_Westervelt}
(p_0, p_1)\in \spaceW_0=\Honethree \times \Honetwo.
\end{equation} 
We claim that problem \eqref{ibvp_linWest} has a unique solution which satisfies an energy bound that does not degenerate as $b \rightarrow 0^+$. {Note that since we are interested in the limit as $b \rightarrow 0^+$, we may focus our attention on an interval $[0, \bar{b})$ for some $\bar{b} >0$ without loss of generality.}
\begin{proposition} \label{Prop:LinWest}
	Assume that the domain $\Omega \subset \R^n$, where $n \in \{1, 2, 3\}$, is bounded and $C^3$ regular. Let $b \in [0, \bar{b})$ and let $T>0$ be a given final time. Assume that the non-degeneracy condition \eqref{non-degeneracy_assumption} and the regularity assumptions \eqref{eq:alphagammaf_reg_Wes} and \eqref{IC_Westervelt} hold. Then there exists a unique  solution $p$ of the problem \eqref{ibvp_linWest},  such that
	\begin{equation}\label{regularity}
	\begin{aligned}
	p \in \, \spaceW =& \,\begin{multlined}[t] L^\infty(0,T; \Honethree) \cap W^{1, \infty}(0,T; \Honetwo) \\
	\cap H^2(0,T; H^1(\Om)).
	\end{multlined}
	\end{aligned}
	\end{equation}
	If $b>0$, then additionally $p \in H^1(0,T; \Honethree)$. On the other hand, if $b=0$, then additionally $p\in W^{2,\infty}(0,T; H^1(\Om))$.\\
	\indent Furthermore, the solution fulfills the estimate
	\begin{equation} \label{LinWest:Main_energy_est}
	\begin{aligned}
	& \begin{multlined}[t] \nLtwoLtwo{\nabla \ptt}^2+\esssup_{t \in (0,T)} \nLtwo{\D \pt (t)}^2\,+\esssup_{t \in (0,T)} \nLtwo{\nabla \D p(t)}^2 
	+b\nLtwoLtwo{\nabla \D \pt}^2
	\end{multlined}\\
	\lesssim&\, \begin{multlined}[t] \nLtwo{\nabla \D p_0}^2+\nLtwo{\D p_1}^2,
	\end{multlined}
	\end{aligned}
	\end{equation}
	where the hidden constant tends to infinity as $T \rightarrow \infty$, but does not depend on $b$.
	
	If additionally $p_1 \in \Honethree$, 
	{$q$, $\alpha \in \spaceW$ with 
		\begin{equation}\label{q0p0}
		q(0)=p_0\mbox{ and }q_t(0)=p_1,
		\end{equation}
		and the compatibility condition 
		\begin{equation}\label{compat_West}
		\ptt(0) =\alpha(0)^{-1} (b \D p_1 +c^2\D p_0 +k q_t(0)p_1)
		\end{equation}
		holds,}  we obtain the enhanced estimate
	\begin{equation} \label{LinWest:Main_energy_est_higher}
	\begin{aligned}
	& \begin{multlined}[t] \nLtwoLtwo{p_{ttt}}^2+\esssup_{t \in (0,T)}\nLtwo{\nabla \ptt(t)}^2+\esssup_{t \in (0,T)} \nLtwo{\D \pt (t)}^2
	\\+\esssup_{t \in (0,T)} \nLtwo{\nabla \D p(t)}^2 
	+b \int_0^{T} \nLtwo{\D \ptt}^2\dt
	+b \int_0^{T} \nLtwo{\nabla \D \pt}^2\dt
	\end{multlined}\\
	\lesssim&\, \begin{multlined}[t] \nLtwo{\nabla \D p_0}^2+\nLtwo{\D p_1}^2+b \nLtwo{\nabla \D p_1}^2.
	\end{multlined}
	\end{aligned}
	\end{equation}
	In case $b=0$, the estimate \eqref{LinWest:Main_energy_est_higher} holds without the additional assumption $p_1\in \Honethree$. {If $b>0$, then additionally $p\in W^{1,\infty}(0,T;\Honethree)$.}
\end{proposition}
\begin{proof}
	We carry out the proof via energy analysis of suitable Galerkin approximations in space. We refer to, for example,~\cite{evans2010partial, roubivcek2013nonlinear} for general details on the method and~\cite{KaltenbacherNikolic} for its use in the analysis of nonlinear acoustic models. Compared to similar energy arguments in the literature, the challenge here is to arrive at a bound that does not degenerate as sound diffusivity approaches zero.\\
	\indent We note that $p \in \spaceW$ implies
	\begin{equation}
	\begin{aligned}
	p \in C_w([0,T]; \Honethree), \qquad \pt \in C_w([0,T]; \Honetwo),
	\end{aligned}
	\end{equation}
	and so the initial conditions are meaningful. \\
	\paragraph{\bf Semi-discretization in space} Take $\{w_i\}_{i \in \N}$ to be the complete set of (smooth) eigenfunctions of $-\Delta$ on $H_0^1(\Om)$.  Fix $n \in \mathbb{N}$ and denote $V_n=\text{span}\{w_1, \ldots, w_n\}$. We seek an approximate solution in the form of
	\begin{equation}
	\begin{aligned}
	p^n=& \, \displaystyle \sum_{i=1}^n \xi_i(t)w_i(x),\\
	\end{aligned}
	\end{equation}
	where the unknown coefficients are $\xi_i:(0,T) \rightarrow \mathbb{R}$, $i \in [1,n]$. We choose the approximate initial data
	\begin{equation}
	\begin{aligned}
	(p^n(0), p^n_t(0))=(p^n_0, p^n_1) \in V_n \times V_n
	\end{aligned}
	\end{equation}
	such that
	\begin{equation}
	\begin{aligned}
	& p^n_0 \rightarrow p_0 \quad \text{ in } \Honethree, \\
	& p^n_1 \rightarrow p_1 \quad \text{ in } \Honetwo,
	\end{aligned}
	\end{equation}
	which is possible since $V_n$ is dense in $\Honethree$.  We then study the following semi-discrete approximation of our problem:
	\begin{equation} \label{ibvp_semi-discrete}
	\begin{aligned} 
	\begin{cases}
	(\ptt^n - \tfrac{b}{\alpha} \D p_t^n - \tfrac{c^2}{\alpha}\D p^n - \tfrac{k}{\alpha} q_t p_t^n,  \phi)_{L^2} = 0, 
	\\[1mm]
	\text{for every $\phi \in V_n$ pointwise a.e. in $(0,T)$}, \\[1mm]
	(p^n(0), \pt^n(0))=(p^n_0, p^n_1).
	\end{cases}
	\end{aligned}
	\end{equation}
	To arrive at a matrix equation, we also introduce matrices $M^n=[M_{ij}]$, $K^n=[K_{ij}]$ and $C^n=[C_{ij}]$ as
	\begin{equation}
	\begin{aligned}
	& M^n_{ij}=( w_i, w_j)_{L^2}, \\
	& K^n_{ij}(t)=-(\tfrac{c^2}{\alpha(t)}\Delta w_i, w_j)_{L^2},\quad C^n_{ij}(t)=-(\tfrac{b}{\alpha(t)}\Delta w_i, w_j)_{L^2}-\prodLtwo{\tfrac{k}{\alpha(t)}q_t(t) w_i}{w_j},\\
	\end{aligned}
	\end{equation}
	for $i,j \in [1,n]$. By setting $\xi^n=[\xi_{1} \ldots \xi_{n}]^T$, $\xi_0^n=[\xi_{1, 0} \ldots \xi_{n, 0}]^T$, and $\xi_1^n=[\xi_{1, 1} \ldots \xi_{n,n}]^T$,  problem \eqref{ibvp_semi-discrete} can be further rewritten as
	\begin{equation} \label{ODE_system}
	\begin{aligned}
	\begin{cases}
	M^n 
	\xi^n_{tt}+C^n(t) \xi_t^n+K^n(t) 
	\xi^n=0, \\
	(\xi^n(0), \xi^n_t(0))=(\xi^n_0, \xi^n_1).
	\end{cases}
	\end{aligned}
	\end{equation}
	After additionally rewriting \eqref{ODE_system} as a first-order system and owing to the fact that the matrix $M^n$ is positive definite, existence of an absolutely continuous solution $[\xi^n, \xi^n_t]^{T}$ on $[0, T_n]$ for some $T_n \leq T$ follows by ODE existence theory; see, for example,~\cite[\S 1]{roubivcek2013nonlinear}. To see that $\xi^n \in H^2(0, T_n)$, we can employ a standard bootstrap argument; cf.~\cite[\S 3]{KaltenbacherNikolic}. We thus conclude that problem \eqref{ibvp_semi-discrete} has a solution $p^n \in H^2(0,T_n; V_n) \subset C^1([0,T_n]; V_n)$. \\
	\paragraph{\bf A priori estimates}
	Our next goal is to derive a uniform in $n$ estimate for the solution of \eqref{ibvp_semi-discrete}. Keeping in mind that we wish to arrive at a higher-order bound, we test the semi-discrete problem with $\phi=\Delta^2 p^n_t \in V_n$ to arrive at
	\begin{equation} \label{identity_1}
	\begin{aligned} 
	(\ptt^n - \tfrac{b}{\alpha} \Delta p_t^n -\tfrac{c^2}{\alpha}\Delta p^n - \tfrac{k}{\alpha} q_t p_t^n,  \Delta^2 p^n_t)_{L^2} = 0,
	\end{aligned}
	\end{equation}
	a.e. in time. Whenever the temporal argument is skipped here and in the following, we suppose the stated (in)equality to hold for almost every time instance in $(0,T_n)$. We have chosen our discrete space so that $\ptt^n= \Delta p^n_t=0$ on $\partial \Om$ and, therefore, we have
	\begin{equation}
	\begin{aligned}
	-\prodLtwo{\ooal \D p^n}{\D^2 \ptn} =&\, \prodLtwo{\D p^n \nabla[\ooal] + \ooal \nabla\D p^n}{\nabla\D \ptn}\\
	=&\,\begin{multlined}[t] \prodLtwo{\D p^n \nabla[\ooal]}{\nabla\D \ptn}
	\\+\frac12\ddt\prodLtwo{\ooal \nabla\D p^n}{\nabla\D p^n} -\frac12 \prodLtwo{[\ooal]_t \nabla\D p^n}{\nabla\D p^n}\end{multlined}\\
	=&\,\begin{multlined}[t]  -\prodLtwo{\nabla\D p^n\cdot\nabla[\ooal] + \D p^n \D[\ooal]}{\D \ptn}
	\\+\frac12\ddt\prodLtwo{\ooal \nabla\D p^n}{\nabla\D p^n} 
	-\frac12 \prodLtwo{[\ooal]_t \nabla\D p^n}{\nabla\D p^n}.\end{multlined}
	\end{aligned}
	\end{equation}
	Furthermore, it holds
	\begin{equation}
	\begin{aligned}
	\D [\ooal q_t\, \ptn] &= 2\ooal\nabla q_t\cdot \nabla \ptn + \ooal  q_t \D \ptn  + \ooal \ptn \D q_t\\
	&\qquad+\D[\ooal] \, q_t \ptn + 2\nabla[\ooal] \cdot\nabla q_t\, \ptn  + 2\nabla[\ooal] \cdot\nabla \ptn\, q_t.
	\end{aligned}
	\end{equation}
	Thus, we arrive at the following energy identity:
	\begin{equation} \label{energy_id}
	\begin{aligned}
	&\frac12\ddt\nLtwo{\D \ptn}^2 + b \nLtwo{\sqrt{\ooal}\nabla\D \ptn}^2 + c^2 \frac12\ddt \nLtwo{\sqrt{\ooal}\nabla\D p^n}^2 \\
	&=k \prodLtwo{2\ooal\nabla q_t\cdot \nabla \ptn + \ooal q_t \D \ptn  + \ooal  \ptn \D q_t 
		+ 2\nabla[\ooal] \cdot\nabla \ptn\, q_t + 2\nabla[\ooal] \cdot\nabla q_t\, \ptn  \\
		&\qquad+\D[\ooal] \, q_t \ptn}{ \D \ptn}\\
	&\quad- b \prodLtwo{\D \ptn \nabla[\ooal]}{\nabla\D \ptn}
	+ c^2 \prodLtwo{\nabla\D p^n\cdot\nabla[\ooal] + \D p^n \D[\ooal]}{\D \ptn}\\
	&\quad+ c^2\frac12 \prodLtwo{[\ooal]_t \nabla\D p^n}{\nabla\D p^n}.
	\end{aligned}
	\end{equation}
	We can rely on the following estimates to handle the $\alpha$ terms on the right:
	\[
	\begin{aligned}
	\nLinf{\ooal}\leq&\, \tfrac{1}{\ulal}, \\
	\nLinf{\nabla[\ooal]}=&\,\nLinf{\alpha^{-2}\nabla\alpha}\leq \tfrac{1}{\ulal^2} |k| \CHtwo \nLtwo{\nabla\D  q}=:\tilde{C}_1\,|k|\nLtwo{\nabla\D  q}\\
	\nLfour{\D [\ooal]}=&\,\nLfour{-2\alpha^{-3}|\nabla\alpha|^2 + \alpha^{-2}\D \alpha},  
	\\
	\leq&\, \tfrac{2}{\ulal^3} |k|^2 \CHtwoeight \nLtwo{\nabla\D  q}^2 
	+ \tfrac{1}{\ulal^2} |k| \CHone \nLtwo{\nabla\D  q}\\
	=&\,: |k|\nLtwo{\nabla\D  q}(\tilde{C}_2+\tilde{C}_3\,|k|\nLtwo{\nabla\D  q})
	\end{aligned}
	\]
	as well as
	\[
	\begin{aligned}
	\nLinf{[\ooal]_t}&=\nLinf{\alpha^{-2}\alpha_t}\leq \tfrac{1}{\ulal^2} |k| \CHtwo \nLtwo{\D  q_t} =:\tilde{C}_4 |k| \nLtwo{\D  q_t}.
	\end{aligned}
	\]
	Combined with Young's inequality, this approach yields the energy estimate
	\begin{equation}\label{enest}
	\begin{aligned}
	&\frac12\ddt\nLtwo{\D  \ptn}^2 + b \frac{1}{\olal}\,\nLtwo{\nabla\D  \ptn}^2 + c^2 \frac{1}{\olal}\,\frac12\ddt \nLtwo{\nabla\D \pn}^2 \\
	\leq&\, \begin{multlined}[t]\left(2\tfrac{|k|}{\ulal} ((\CHone)^2+\CHtwo) \right. \\ \left.
	+ |k|\nLtwo{\nabla\D  q} \left(4 \tilde{C}_1 \CPF \CHtwo
	+ \bigl(\tilde{C}_2+\tilde{C}_3\,|k|\nLtwo{\nabla\D  q}\bigr)\CHtwoeight\right) \right)\\
	\times\nLtwo{\D  q_t}\, \nLtwo{\D  \ptn}^2\\
	+ \frac{b\olal}{2} \tilde{C}_1^2\,|k|^2\nLtwo{\nabla\D  q}^2 \nLtwo{\D \ptn}^2+ \frac{b}{2\olal}\nLtwo{\nabla\D \ptn}^2\\
	+ c^2 
	|k|\nLtwo{\nabla\D  q} \bigl(\tilde{C}_1+(\tilde{C}_2+\tilde{C}_3\,|k|\nLtwo{\nabla\D  q}) \CHone\bigr)\\
	\times\nLtwo{\nabla\D \pn}\, \nLtwo{\D  \ptn}\\
	+ c^2 \tilde{C}_4 |k| \nLtwo{\D  q_t} \nLtwo{\nabla\D  \pn}^2. \end{multlined}
	\end{aligned}
	\end{equation}
	We next estimate $\nLtwo{\nabla \pttn}$. Since $\pn$ satisfies the semi-discrete equation, we have
	\begin{equation}
	\begin{aligned}
	\nLtwo{\nabla \pttn}^2=&\, \prodLtwo{\nabla \pttn}{\nabla \pttn} 
	=-\prodLtwo{\pttn}{\Delta \pttn}\\
	=&\, -\prodLtwo{\alpha^{-1}(b \Delta p_t^n +c^2\Delta p^n + k q_t p_t^n)}{\Delta \pttn}.
	\end{aligned}
	\end{equation}
	Therefore, 
	\begin{equation} \label{LinWest:id_ptnn}
	\begin{aligned}
	\nLtwo{\nabla \pttn}^2=&\, \begin{multlined}[t]\prodLtwo{\alpha^{-1}(b \nabla \Delta p_t^n +c^2 \nabla \Delta p^n + k \nabla [q_t p_t^n])}{\nabla \pttn} \\[1mm]
	+\prodLtwo{\nabla [\alpha^{-1}](b \Delta p_t^n +c^2 \Delta p^n + k q_t p_t^n)}{\nabla \pttn}. \end{multlined}
	\end{aligned}
	\end{equation}
	The above identity together with the Cauchy--Schwarz inequality allow us to conclude that
	\begin{equation}\label{LinWest:bound_ptnn}
	\begin{aligned}
	&\nLtwo{\nabla \pttn}\\
	\leq&\, \begin{multlined}[t]\ulal^{-1} (b\nLtwo{\nabla \Delta p_t^n} +c^2 \nLtwo{\nabla \Delta p^n} +| k| \nLtwo{\nabla q_t}\CHtwo \nLtwo{\D p_t^n}\\
	+ |k|\CHone^2 \nLtwo{\nabla q_t} \nLtwo{\Delta \ptn})\\
	+\ulal^{-2}|k|\CHtwo\nLtwo{\nabla \Delta q} (b \nLtwo{\Delta p_t^n} +c^2 \nLtwo{\Delta p^n}\\ + |k |\nLtwo{q_t}\CHtwo \nLtwo{\D p_t^n}).
	\end{multlined}
	\end{aligned}
	\end{equation}
	Applying Gronwall's inequality to \eqref{enest} and taking the supremum over $t \in (0,T_n)$ yield
	\begin{equation} 
	\begin{aligned}
	& \begin{multlined}[t] \sup_{t \in (0,T_n)} \nLtwo{\D \ptn(t)}^2\,+\sup_{t \in (0,T_n)} \nLtwo{\nabla \D \pn(t)}^2 
	+b 
	\int_0^{T_n} \nLtwo{\nabla \D \ptn}^2\dt
	\end{multlined}\\
	\leq&\, \begin{multlined}[t] C(T)(\nLtwo{\nabla \D p_0}^2+\nLtwo{\D p_1}^2),
	\end{multlined}
	\end{aligned}
	\end{equation}
	where we have additionally bounded the approximate initial data by the exact data in the same norm. The constant above is given by
	\begin{equation}\label{CofT}
	\begin{aligned}
	C(T)= C_1\exp(C_2\|k|(\|q\|_{X^W_\alpha}+(1+b^2)\|q\|_{X^W_\alpha}^2+\|q\|_{X^W_\alpha}^3) T),
	\end{aligned}
	\end{equation}
	where $C_1$, $C_2>0$ do not depend neither on $b$ nor on 
	{$T_n$,}  
	$n$. Thanks to \eqref{LinWest:bound_ptnn}, we then also have a uniform bound on $\pttn$:
	\begin{equation}\label{boundnablaptt}
	\begin{aligned}
	\int_0^{T_n} \nLtwo{\nabla \pttn}^2 \dt
	\lesssim&\, \begin{multlined}[t]b^2\nLtwoLtwo{\nabla \Delta p_t^n} + \nLtwoLtwo{\nabla \Delta p^n}^2 
	\\+ (1+\nLinfLtwo{\nabla q_t}^2) \nLtwoLtwo{\Delta \ptn}^2\\
	+\nLinfLtwo{\nabla \Delta q}^2 \left(b^2 \nLtwoLtwo{\Delta p_t^n}^2 + \nLtwoLtwo{\Delta p^n}^2 \right.\\ \left.+ \nLinfLtwo{q_t}^2 \nLtwoLtwo{\D p_t^n}^2 \right).
	\end{multlined}
	\end{aligned}
	\end{equation}
	If $b=0$, we even get a bound on $\sup_{(0,T_n)} \nLtwo{\nabla \pttn}$ here.
	
	\indent The uniform bounds with respect to $T_n$ allow us to extend the existence interval to $[0,T]$. \\ 
	\paragraph{\bf Passing to the limit} Thanks to the uniform bounds on $\pn$  and the fact that our spatial and temporal domains are bounded, we can employ standard compactness arguments and conclude that there exist a subsequence, which we do not relabel, and a function $p$ such that
	\begin{equation} \label{weak_limits}
	\begin{alignedat}{4} 
	p_{tt}^n  &\relbar\joinrel\rightharpoonup p_{tt} &&\text{ weakly}  &&\text{ in } &&L^2(0,T;H_0^1(\Omega)),  \\
	p_t^n &\relbar\joinrel\rightharpoonup p_t &&\text{ weakly-$\star$} &&\text{ in } &&L^\infty(0,T; H_0^1(\Omega)\cap H^2(\Omega)),\\
	p_t^n &\relbar\joinrel\rightharpoonup p_t &&\text{ weakly} &&\text{ in } &&L^2(0,T; \Honethree)
	\ \mbox{ if }b>0,\\
	p^n &\relbar\joinrel\rightharpoonup p &&\text{ weakly-$\star$} &&\text{ in } &&L^\infty(0,T; \Honethree).
	\end{alignedat} 
	\end{equation}
	We next wish to prove that this limit solves \eqref{ibvp_linWest}. We multiply the semi-discrete equation by $\eta \in C_0^\infty(0,T)$ and integrate over time to obtain 
	\begin{equation} \label{1}
	\begin{aligned}
	\int_{0}^T(p^n_{tt}, w_i)_{L^2}\, \eta(t)\, \textup{d}t
	=\, \int_{0}^T(c^2\Delta \pn+b\Delta \ptn +kq_t \ptn, \ooal w_i)_{L^2} \eta(t) \, \textup{d}t
	\end{aligned}
	\end{equation}
	for $i \in [1, n]$. Letting $n \rightarrow \infty$ in the above equation and employing \eqref{weak_limits} leads to
	\begin{equation}
	\begin{aligned}
	\int_{0}^T(p_{tt}, w_i)_{L^2}\, \eta(t)\, \textup{d}t
	=\, \int_{0}^T(c^2\Delta p+b\Delta \pt +k q_t \pt, \ooal w_i)_{L^2} \eta(t) \, \textup{d}t,
	\end{aligned}
	\end{equation}
	for all  $i \in \mathbb{N}$ and $\eta \in C^\infty_0(0,T)$. Since $\displaystyle \bigcup_{n \in \mathbb{N}}V_n$ is dense in $L^2(\Om)$, and by the Fundamental Lemma of Calculus of Variations, $p$ solves the original problem. {For $b \neq 0$,} due to the embedding
	\begin{equation}
	\begin{aligned}
	&	p^n \in H^{1}(0,T; \Honethree) \hookrightarrow C([0,T]; \Honethree),
	\end{aligned}
	\end{equation}
	we know that $p^n(0) \rightharpoonup p(0) \quad \text{in } \Honethree$. On the other hand, $\pn(0) \rightarrow p_0$ in $\Honethree$, and so $p(0)=p_0$. We also have 
	\begin{equation}
	\begin{aligned}
	p^n_t \in&\, \begin{multlined}[t] L^\infty(0,T; \Honetwo) \cap H^1(0,T; H_0^1(\Om)) \\ \hookrightarrow C_w([0,T]; \Honetwo) \cap C([0,T]; H_0^1(\Om)), 
	\end{multlined}
	\end{aligned}
	\end{equation}
	and, therefore, also $p_t(0)=p_1$.\\
	\paragraph{\bf Energy estimate for \mathversion{bold}$p$} We can take the limit inferior as $n \rightarrow \infty$ of the discrete energy bound, and by virtue of the weak and the weak-$\star$ lower semi-continuity of norms obtain the final estimate. Uniqueness of a solution in $\spaceW$ follows by {testing the homogeneous problem (i.e., with $p_0=p_1=0)$ with $p_t$ and relying on the bound
		\begin{equation} \label{lin_lower_bound}
		\begin{aligned}
		&\frac12 \|\sqrt{\alpha(t)}p_t(t)\|^2_{L^2}+c^2\|\nabla p(t)\|^2_{L^2} +b \int_0^t\|\nabla p_t\|^2\dxs\\
		=&\frac12 \int_0^t(\alpha_t p_t, p_t)_{L^2}\ds+k\int_0^t (q_t p_t, p_t)\ds\\
		\lesssim &\, (\|\alpha_t\|_{L^\infty(L^\infty)}+\|q_t\|_{L^\infty(L^\infty)})\|p_t\|^2_{L^2(L^2)},
		\end{aligned}
		\end{equation}
		for $t \in [0,T]$, together with Gronwall's inequality.}\\ 
	\paragraph{\bf Additional regularity}
	The already established existence of the solution $p\in\spaceW$, via the PDE implies \[\ptt=\ooal(b\D\pt+c^2\D p+kq_t\pt)\in H^1(0,T;H^{-1}(\Omega))\]
	and, therefore, $p\in H^3(0,T;H^{-1}(\Omega))\cap \spaceW$. 
	Since by interpolation (with interpolation parameter $\theta=\frac{1-m}{2-m}\in[0,\frac12)$, where $m=\min\{s,1\}>0$ and $r=\frac{m}{2(2-m)}>0$) we have 
	\[
	\begin{aligned}
	H^3(0,T;H^s(\Omega))\cap H^2(0,T;\Honetwo)&\hookrightarrow H^{2+(1-\theta)}(0,T;H^{m+\theta(2-m)}(\Omega))
	\\&=H^{5/2+r}(0,T;\Hone)\hookrightarrow C^2([0,T];\Hone),
	\end{aligned}
	\] 
	the imposed initial conditions on $p$, $p_t$, and $p_{tt}$ make sense in the spaces $\Hthree$, $\Hthree$, and $\Hone$, respectively.
	
	We now derive a uniform bound on $\nLinfLtwo{\nabla \ptt}$ also in the case $b>0$ for sufficiently regular initial data. Note that we have already established such a bound when $b=0$, right after estimate \eqref{boundnablaptt}. \\
	\indent To justify the derivation of this bound, {we will employ a Galerkin discretization of the time-differentiated version of \eqref{ibvp_linWest}:
		\begin{equation}\label{Galerkintimediff}
		\begin{aligned} 
		\begin{cases}
		(\alpha \tilde{p}^n_{ttt} +\alpha_t \tilde{p}^n_{tt}- b \D \tilde{p}^n_{tt} - c^2\D \tilde{p}^n_t - k q_t \tilde{p}^n_{tt} - {k q_{tt}\tilde{p}^n_t},  \phi)_{L^2} = 0, 
		\\[1mm]
		\text{for every $\phi \in V_n$ pointwise a.e. in $(0,T)$}, \\[1mm]
		(\tilde{p}^n(0), \tilde{p}^n_t(0), \tilde{p}^n_{tt}(0))=(\tilde{p}^n_0, \tilde{p}^n_1, \tilde{p}^n_2),
		\end{cases}
		\end{aligned}
		\end{equation}
		where $\tilde{p}^n_0$ , $\tilde{p}^n_1$, $\tilde{p}^n_2$ are the Galerkin projections of $p_0$, $p_1$, and $\alpha(0)^{-1} (b \D p_1 +c^2\D p_0 +k q_t(0)p_1)$ (for the latter see the compatibility condition \eqref{compat_West}).} To relate the solution of the time-differentiated equation to the original one, the compatibility condition \eqref{compat_West} is needed. \\
	Testing 
	\eqref{Galerkintimediff}
	with $\tilde{p}^n_{ttt}-\Delta \tilde{p}^n_{tt}$ results in
	\begin{equation} 
	\begin{aligned} 
	(\alpha \tilde{p}^n_{ttt}+\alpha_t \tilde{p}^n_{tt} - b \Delta \tilde{p}^n_{tt} -c^2\Delta \tilde{p}^n_t - k q_t \tilde{p}^n_{tt}-k q_{tt}\tilde{p}^n_t,  \tilde{p}^n_{ttt}-\Delta \tilde{p}^n_{tt})_{L^2} = 0.
	\end{aligned}
	\end{equation}
	We can then rely on the identity
	\begin{equation}
	\addtolength{\jot}{0.5mm}
	\begin{aligned}
	\prodLtwo{\alpha \tilde{p}^n_{ttt}}{\tilde{p}^n_{ttt}-\Delta \tilde{p}^n_{tt}} 
	=&\,\prodLtwo{\alpha \tilde{p}^n_{ttt}}{\tilde{p}^n_{ttt}}+\prodLtwo{\nabla[\alpha \tilde{p}^n_{ttt}]}{\nabla \tilde{p}^n_{tt}} \\
	=&\,\prodLtwo{\alpha \tilde{p}^n_{ttt}}{\tilde{p}^n_{ttt}}+ \prodLtwo{\alpha \nabla \tilde{p}^n_{ttt}+  \tilde{p}^n_{ttt}\nabla \alpha}{\nabla \tilde{p}^n_{tt}}\\
	=&\,\begin{multlined}[t]\prodLtwo{\alpha \tilde{p}^n_{ttt}}{\tilde{p}^n_{ttt}}+\frac12 \ddt (\alpha \nabla \tilde{p}^n_{tt}, \nabla \tilde{p}^n_{tt})\\ -\frac12 (\alpha_t \nabla \tilde{p}^n_{tt},\nabla \tilde{p}^n_{tt})
	+( \tilde{p}^n_{ttt} \nabla \alpha,\nabla \tilde{p}^n_{tt}),\end{multlined}
	\end{aligned}
	\end{equation}
	and, similarly, the identity
	\begin{equation}
	\addtolength{\jot}{0.5mm}
	\begin{aligned}
	\prodLtwo{\alpha_t \tilde{p}^n_{tt}}{\tilde{p}^n_{ttt}-\Delta \tilde{p}^n_{tt}} 
	=&\,\prodLtwo{\alpha_t \tilde{p}^n_{tt}}{\tilde{p}^n_{ttt}}+\prodLtwo{\nabla[\alpha_t \tilde{p}^n_{tt}]}{\nabla \tilde{p}^n_{tt}} \\
	=&\,\begin{multlined}[t]\prodLtwo{\alpha_t \tilde{p}^n_{tt}}{\tilde{p}^n_{ttt}}+\prodLtwo{\alpha_t \nabla \tilde{p}^n_{tt}}{\nabla \tilde{p}^n_{tt}}+\prodLtwo{\tilde{p}^n_{tt} \nabla \alpha_t}{\nabla \tilde{p}^n_{tt}}.\end{multlined}
	\end{aligned}
	\end{equation}
	In this manner, we obtain 
	\begin{equation} \label{id_additionalreg}
	\begin{aligned}
	&\begin{multlined}[t]\nLtwo{\sqrt{\alpha}  \tilde{p}^n_{ttt}}^2+\frac12 \ddt \nLtwo{\sqrt{\alpha} \nabla \tilde{p}^n_{tt}}^2\\+ \frac12 c^2 \ddt\nLtwo{\D \tilde{p}^n_t}^2+b \nLtwo{\D \tilde{p}^n_{tt}}^2
	+ \frac12 b \ddt\nLtwo{\nabla \tilde{p}^n_{tt}}^2\end{multlined}\\
	=&\, \begin{multlined}[t]\frac12 \prodLtwo{ \alpha_t \nabla \tilde{p}^n_{tt}}{\nabla \tilde{p}^n_{tt}}
	-\prodLtwo{\tilde{p}^n_{ttt} \nabla \alpha}{\nabla \tilde{p}^n_{tt}}-\prodLtwo{\alpha_t \tilde{p}^n_{tt}}{\tilde{p}^n_{ttt}}\\-\prodLtwo{\tilde{p}^n_{tt} \nabla \alpha_t}{\nabla \tilde{p}^n_{tt}}
	+ k\prodLtwo{q_t \tilde{p}^n_{tt}+ q_{tt}\tilde{p}^n_t}{\tilde{p}^n_{ttt}-\Delta \tilde{p}^n_{tt}} 
	+c^2\prodLtwo{\D \tilde{p}^n_t}{\tilde{p}^n_{ttt}} =: \textup{rhs}. \end{multlined}
	\end{aligned}
	\end{equation}
	Furthermore, we have
	\begin{equation}
	\begin{aligned}
	\prodLtwo{q_t \tilde{p}^n_{tt}+ q_{tt}\tilde{p}^n_t}{\tilde{p}^n_{ttt}-\Delta \tilde{p}^n_{tt}}=&\, \prodLtwo{q_t \tilde{p}^n_{tt}+ q_{tt}\tilde{p}^n_t}{\tilde{p}^n_{ttt}}+\prodLtwo{\nabla (q_t \tilde{p}^n_{tt}+ q_{tt}\tilde{p}^n_t)}{\nabla\tilde{p}^n_{tt}}. \\
	\end{aligned}
	\end{equation}
	Thus, we can estimate the right-hand side of \eqref{id_additionalreg} as follows:
	\begin{equation}
	\begin{aligned}
	|\textup{rhs}| \lesssim&\, \begin{multlined}[t] (1+\nLinf{\alpha_t})\nLtwo{\nabla \tilde{p}^n_{tt}}^2+\varepsilon \nLtwo{\tilde{p}^n_{ttt}}^2+(1+\nLinf{\nabla \alpha}^2+\nLinf{q_t})\nLtwo{\nabla \tilde{p}^n_{tt}}^2\\ 
	+ (\nLthree{\alpha_t}^2+\nLthree{\nabla \alpha_t}^2+\nLthree{q_t}^2+\nLthree{\nabla q_t}^2)\nLsix{\tilde{p}^n_{tt}}^2 \\
	+\nLthree{q_{tt}}^2\nLsix{\nabla\tilde{p}^n_t}^2+\nLthree{q_{tt}}^2\nLsix{\tilde{p}^n_t}^2+\nLtwo{\nabla q_{tt}}^2\nLinf{\tilde{p}^n_t}^2
	+\nLtwo{\D \tilde{p}^n_t}^2 \end{multlined}
	\end{aligned}
	\end{equation}
	for all $t \in [0,T]$, with $\varepsilon$ small enough, but independent of $b$. From the compatibility condition \eqref{compat_West} and boundedness of the projection operator we can obtain a bound on $\nabla \tilde{p}^n_{tt}(0)$ as follows:  
	\begin{equation}
	\begin{aligned}
	\nLtwo{\nabla \tilde{p}^n_{tt}(0)} \lesssim&\, \begin{multlined}[t] b\nLtwo{\nabla \Delta p_1} + \nLtwo{\nabla \Delta p_0} + \nLtwo{\nabla q_t(0)} \nLtwo{\D p_1}
	+  \nLtwo{\nabla q_t(0)} \nLtwo{\Delta p_1}\\
	+\nLtwo{\nabla \Delta q(0)} (b \nLtwo{\Delta p_1} +\nLtwo{\Delta p_0}+ \nLtwo{q_t(0)} \nLtwo{\D p_1}).
	\end{multlined}
	\end{aligned}
	\end{equation}
	{Applying Gronwall's inequality thus yields the energy estimate
		\begin{equation} \label{enest_additionalreg}
		\begin{aligned}
		&\begin{multlined}[t]\nLtwo{\sqrt{\alpha}  \tilde{p}^n_{ttt}}^2+\frac12 \ddt \nLtwo{\sqrt{\alpha} \nabla \tilde{p}^n_{tt}}^2\\+ \frac12 c^2 \ddt\nLtwo{\D \tilde{p}^n_t}^2+b \nLtwo{\D \tilde{p}^n_{tt}}^2
		+ \frac12 b \ddt\nLtwo{\nabla \tilde{p}^n_{tt}}^2\end{multlined}\\
		\lesssim& \begin{multlined}[t] \nLtwo{\nabla \D p_0}^2+\nLtwo{\D p_1}^2+b \nLtwo{\nabla \D p_1}^2,
		\end{multlined}
		\end{aligned}
		\end{equation}
		where we have used that also \eqref{q0p0} holds.\\
		\indent Due to the uniform bound \eqref{enest_additionalreg}, $\tilde{p}^n$ has a weakly convergent subsequence whose limit $\tilde{p}$ inherits this bound. Moreover, $\tilde{p}$ solves the time-differentiated PDE and satisfies the compatibility condition \eqref{compat_West}. By integration with respect to time, it therefore satisfies the original initial-value problem \eqref{ibvp_linWest} and thus by uniqueness has to coincide with $p$.
	}
	
	Combined with the previous uniform estimate for smooth approximations, this allows us to obtain the higher-order in time bound
	\eqref{LinWest:Main_energy_est_higher}. 
\end{proof}

\section{Uniform bounds for the Westervelt equation} \label{Sec:West}
We next analyze the Westervelt equation by introducing the fixed-point mapping
\begin{equation}
\mathcal{T}:q\mapsto p,
\end{equation}
where $q$ will belong to a suitably chosen ball in the space \[ \spaceW \cap \LinfLinf, \]
with $\spaceW$ defined in \eqref{regularity}. The function $p$ will solve the linearized Westervelt equation:
\begin{equation} \label{linWest_fp}
\alpha \ptt - b \D p_t - c^2\D p - k q_t p_t =0,
\end{equation}
with $\alpha=1-kq$ and initial conditions $p(0)=q(0)=p_0$, $p_t(0)=q_t(0)=p_1$.\\
\indent We first determine the conditions under which $\mathcal{T}$ is a self-mapping on an appropriately chosen set $M$. More precisely, we define $M$ as a subset of the solution space $\spaceW$,
	\begin{equation}\label{defM}
	\begin{aligned}
	M = \left\{\vphantom{\frac{1}{2|k|}}	q\in \spaceW:\right.&\,   q(0)=p_0, \ q_t(0)=p_1, \
   \nLinfLinf{q}\leq \frac{1}{2|k|} \\
	&  \hspace*{-1.3cm}\left. \nLtwoLtwo{\nabla q_{tt}}^2+\nLinfLtwo{\D  q_t}^2 + \nLinfLtwo{\nabla\D  q}^2
+b\nLtwoLtwo{\nabla\D q_t}^2\leq R^2 \vphantom{\frac{1}{2|k|}} \right\}. 
	\end{aligned}
	\end{equation}
The imposed conditions allow us to uniformly bound the hidden constant in estimate \eqref{LinWest:Main_energy_est} by 
\begin{equation}\label{Clin_West}
C_{\textup{lin}}(T,R)= C_1 \exp{\left(C_2|k|(R+R^2+R^3)T\right)}, 
\end{equation}  
where the positive constants $C_1$ and $C_2$ do not depend on $b \in [0, \bar{b})$ for fixed $\bar{b}$ nor on $R$, $T$; cf. \eqref{CofT}.
\begin{proposition} \label{Prop:SelfMapping}
Let $\Omega \subset \R^n$, where $n \in \{1, 2, 3\}$, be bounded and $C^3$ regular. Furthermore, let $b \in [0, \bar{b})$ and let $T>0$ be a given final time. For initial data satisfying the regularity condition \eqref{IC_Westervelt}, the mapping $\mathcal{T}:q\mapsto p$  is a self-mapping on the set $M$ defined in \eqref{defM}, provided the $H_0^1(\Omega) \times L^2(\Omega)$ norm of the initial data is small enough so that
\begin{equation}\label{smallnessinit}
{C_{\textup{A}}\tilde{C}_1 \exp(\tilde{C}_2RT)(\|p_1\|^2_{L^2}+\|\nabla p_0\|^2_{L^2})^{1/4}R^{1/2} \leq \frac{1}{2|k|}}
\end{equation}
and the radius $R$ and final time $T$ are chosen such that
\begin{equation}\label{smallnessT}
C_1 \exp{\left(C_2|k|(R+R^2+R^3)T\right)}(\nLtwo{\nabla \D p_0}^2+\nLtwo{\D p_1}^2)\leq R^2.
\end{equation}
Furthermore, $p=\mathcal{T}(q)$ satisfies the energy estimate
\begin{equation}\label{estpnonlin}
\begin{aligned}
& \nLtwoLtwo{\nabla \ptt}^2+\esssup_{t \in (0,T)} \nLtwo{\D \pt (t)}^2\,+\esssup_{t \in (0,T)} \nLtwo{\nabla \D p(t)}^2 
+b\nLtwoLtwo{\nabla \D \pt}^2
\\
\le&\, C_{\textup{lin}}(T,R) (\nLtwo{\nabla \D p_0}^2+\nLtwo{\D p_1}^2).
\end{aligned}
\end{equation}
{
Here the constants $C_{\textup{A}}$, $\tilde{C}_1$, $\tilde{C}_2$, $C_1$, $C_2$, and $C_{\textup{lin}}(T,R)$ are as in \eqref{Clin_West} and \eqref{CAC1tilC2til} below.
}
\end{proposition}
\begin{remark}
To satisfy condition \eqref{smallnessinit}, only smallness of the initial data in a weaker norm is needed; namely {$H_0^1(\Omega) \times L^2(\Omega)$} rather than the $\Honethree\times\Honetwo$ norm required for the higher-order energy estimates. 
Condition \eqref{smallnessT} can be satisfied even for initial data with large {higher-order} norm, by, for example, choosing 
\[R^2= 2 C_1 (\nLtwo{\nabla \D p_0}^2+\nLtwo{\D p_1}^2)\] and assuming the final time to be short enough, so that \[T\leq\frac{\ln(2)}{C_2|k|(R+R^2+R^3)}.\]
Note that both smallness conditions \eqref{smallnessinit} and \eqref{smallnessT} are additionally mitigated by the fact that the nonlinearity parameter $k$ is typically small in magnitude in ultrasonic applications; cf.~\cite[\S 5]{kaltenbacher2007numerical}.
\end{remark}
\begin{proof}
We rely in the proof on our previous linear analysis, which, in particular, implies $M\not=\emptyset$ and that $\mathcal{T}$ is well-defined on $M$. Let $q \in M$. The bound \[\nLinfLinf{q}\leq \frac{1}{2|k|}\] immediately implies that
the non-degeneracy assumption \eqref{non-degeneracy_assumption} holds with $\ulal=\frac12$ and $\olal=\frac32$. 
Moreover, the regularity assumptions of Proposition~\ref{Prop:LinWest} are satisfied with the uniform bounds
\begin{equation}
\begin{aligned}
&\nLinfHthree{\alpha}\leq |\Omega|^{1/2}+|k| C_{(-\Delta)^{-1}} R, \qquad 
\nLinfHtwo{\alpha_t}\leq |k| C_{(-\Delta)^{-1}} R, \\
&\nLinfHtwo{q_t}\leq C_{(-\Delta)^{-1}} R\,.
\end{aligned}
\end{equation}
Hence Proposition~\ref{Prop:LinWest} applied to $p=\mathcal{T}(q)$, together with the imposed smallness condition
\eqref{smallnessT} imply that $\mathcal{T}$ is well-defined and yield the $R^2$ bound in \eqref{defM} on $\mathcal{T}q=p$. \\
\indent To also achieve the $\frac{1}{2|k|}$ bound on the $L^\infty(0,T; L^\infty(\Om))$ norm of $p$, we derive a simple energy estimate for the linearization.
{Similarly to \eqref{lin_lower_bound}, by multiplying equation \eqref{linWest_fp} with $p_t$ and integrating over space and time, we arrive at	
	\begin{equation}
	\begin{aligned}
&\frac12 \left\{\|\sqrt{\alpha(t)}p_t(t)\|^2_{L^2}+c^2\|\nabla p(t)\|^2_{L^2} \right\}\Big \vert_0^t+b \int_0^t\|\nabla p_t\|^2\dxs\\
	\lesssim &\, (\|\alpha_t\|_{L^\infty(L^\infty)}+\|q_t\|_{L^\infty(L^\infty)})\|p_t\|^2_{L^2(L^2)}.   
	\end{aligned}
	\end{equation}
From here by Gronwall's inequality, we have
	\begin{equation}
\begin{aligned}
&\|p_t(t)\|^2_{L^2}+\|\nabla p(t)\|^2_{L^2} 
\leq \, C_1 \exp(C_2RT)(\|p_1\|^2_{L^2}+\|\nabla p_0\|^2_{L^2})  
\end{aligned}
\end{equation}
for all $t \in [0,T]$.  Following, e.g.,~\cite[Theorem 1.4]{bongarti2021vanishing}, the desired $L^\infty$ bound on $p$ can be obtained by employing the above estimate and Agmon's interpolation inequality~\cite[Lemma 4.10]{constantin1988navier}:
\begin{equation} \label{Agmon}
\|p(t)\|_{L^\infty} \leq C_{\textup{A}} \|p(t)\|_{H^1}^{1/2}\|p(t)\|_{H^2}^{1/2}.
\end{equation}
Indeed, by also using that $\|p\|_{L^\infty(H^2)} \lesssim R$, we have 
\begin{equation}\label{CAC1tilC2til}
\begin{aligned}
\|p\|_{L^\infty(L^\infty)}\leq C_{\textup{A}}\tilde{C}_1 \exp(\tilde{C}_2RT)(\|p_1\|^2_{L^2}+\|\nabla p_0\|^2_{L^2})^{1/4}R^{1/2}.
\end{aligned}
\end{equation}
for some positive constants $\tilde{C}_1$ and $\tilde{C}_2$, independent of $b$. Thus for $\mathcal{T}(M) \subset M$ to hold, we need to impose the following condition:
\[
C_{\textup{A}}\tilde{C}_1 \exp(\tilde{C}_2RT)(\|p_1\|^2_{L^2}+\|\nabla p_0\|^2_{L^2})^{1/4}R^{1/2} \leq \frac{1}{2|k|},
\] 
as claimed.}
\end{proof}
\noindent We next prove that $\mathcal{T}$ is a contraction in a suitably chosen topology. 
\begin{proposition}\label{Prop:Contraction}
Let the assumptions on $\Omega$ from Proposition \ref{Prop:SelfMapping} be satisfied.
There exist constants $C_5$, $C_6>0$ depending only on the domain $\Omega$ and the constant $c^2$, but neither on $T$ nor on $b$, such that if either  
\begin{equation}\label{small1}
C_5 \exp{(C_6 T)} k^2 \, (T+1)\, (\|p_0\|^2_{H^3}+\|p_1\|^2_{H^2}) < 1
\end{equation}
or $p_1 \in \Honethree$, \eqref{compat_West}, $q$, $\alpha \in \spaceW$ with \eqref{q0p0}, and
\begin{equation}\label{small2}
C_5 \exp{(C_6 T)} k^2 \, T\, (\|p_0\|^2_{H^3}+\|p_1\|^2_{H^3}) < 1,
\end{equation}
then the mapping $\mathcal{T}$ is a contraction on $M$ in  the topology induced by
	\begin{equation} \label{X_norm}
	\|p\|^2_{X}= \nLtwoLtwo{\ptt}^2+\nLinfLtwo{\nabla \pt}^2 + \nLinfLtwo{\Delta p}^2.
	\end{equation}	
\end{proposition}
Fulfillment of condition ``\eqref{small1} or \eqref{small2}" can be achieved by either imposing smallness on the initial data in $\Honethree\times\Honetwo$ for a fixed given final time $T$, or by imposing short enough final time $T$ for fixed given initial data in $\Honethree\times\Honethree$. For the latter, in case of $b=0$, it suffices to have $\Honethree\times\Honetwo$ regular data. 	
\begin{proof}
In order to prove strict contractivity, we take $q^{(1)}$ and $q^{(2)}$ in $M$. Let then $p^{(1)}=\mathcal{T} q^{(1)}$ and $p^{(2)}=\mathcal{T} q^{(2)} $. Furthermore, we introduce the differences $\overline{p}=p^{(1)} -p^{(2)} $ and $\overline{q}= q^{(1)} -q^{(2)} $. Then we know that $\op$ solves the equation
\begin{equation} \label{West_contract_eq}
(1-k q^{(1)} )\op_{tt}-c^2 \Delta \op-b \Delta \op_t=k q_t^{(1)}\op_t+\underbrace{k\overline{q}_tp^{(2)}_t+k \overline{q}p^{(2)} _{tt}}_{:=f}
\end{equation}
and has zero initial conditions. \\
\indent Testing equation \eqref{West_contract_eq} with $-\Delta \op_t$, integrating over space and $(0,t)$, and noting that $\op(0)=\op_t(0)=0$ yields
\begin{equation}
\begin{aligned}
&\frac{1}{2}\nLtwo{\nabla \op_t(t)}^2+ \frac{c^2}{2} \nLtwo{\Delta \op(t)}^2+b \int_0^t \nLtwo{\D \op_t}^2 \ds \\
=&\,\begin{multlined}[t] k \intTO \nabla (q^{(1)}\op_{tt}) \cdot \nabla \op_t \dx \textup{d}s+ k \intTO \nabla (q^{(1)}_t\op_{t}) \cdot \nabla \op_t \dx \textup{d}s\\
+ \intTO \nabla f \cdot \nabla \op_t \dx \textup{d}s \end{multlined}\\
=&\, \,\begin{multlined}[t] \frac12 k \intO q^{(1)}(t) |\nabla \op_{t}(t)|^2\dx- \frac12 k \intTO q^{(1)}_t |\nabla \op_{t}|^2 \dx \textup{d}s\\  +k \intTO \op_{tt} \nabla q^{(1)} \cdot \nabla \op_t \dx \textup{d}s
+ k \intTO \nabla (q^{(1)}_t\op_{t}) \cdot \nabla \op_t \dx \textup{d}s\\
+ \intTO \nabla f \cdot \nabla \op_t \dx \textup{d}s. \end{multlined}
\end{aligned}
\end{equation}
From here, we further have by Young's inequality
\begin{equation} \label{est_1}
\begin{aligned}
&\frac{1}{2}\ulal \nLtwo{\nabla \op_t(t)}^2+ \frac{c^2}{2} \nLtwo{\Delta \op(t)}^2 +b \nLtwoLtwo{\Delta \op_t}^2 \\
\leq&\, \,\begin{multlined}[t] \frac12|k| \|q^{(1)}_t\|_{L^\infty (L^\infty)}  \nLtwoLtwo{\nabla \op_{t}}^2  + |k| \nLtwoLtwo{\op_{tt}} \nLinfLinf{\nabla q^{(1)}} \nLtwoLtwo{\nabla \op_t}\\
+ |k| \|\nabla q^{(1)}_t\|_{L^\infty(L^3)} \|\op_{t}\|_{L^2 (L^6)} \|\nabla \op_t \|_{L^2(L^2)}
+  \nLtwoLtwo{\nabla f} \nLtwoLtwo{\nabla \op_t}. \end{multlined}
\end{aligned}
\end{equation}
Note that we can derive a bound on the $\nabla f$ term as follows:
\begin{equation}\label{estf}
\begin{aligned}
\nLtwoLtwo{\nabla f}^2
=&\, k^2 \intTO |\nabla (\oq_t p^{(2)}_t+ \oq p^{(2)} _{tt}) |^2 \dx \textup{d}s  \\
\leq&\,\begin{multlined}[t] 2 k^2(\nLinfLtwo{\nabla \oq_t}^2 \nLtwoLinf{p^{(2)}_t}^2+\nLinfLthree{\oq_t}^2\nLtwoLsix{\nabla p^{(2)}_t}^2\\
+ \nLinfLthree{\nabla \oq}^2 \nLtwoLsix{p^{(2)}_{tt}}^2+\nLinfLinf{\oq}^2\nLtwoLtwo{\nabla p^{(2)}_{tt}}^2 ).\end{multlined}
\end{aligned}
\end{equation}
Testing  \eqref{West_contract_eq} additionally with ${\op_{tt}}$ produces
\begin{equation}
\begin{aligned}
&\ulal \|\op_{tt}\|^2_{L^2 L^2} +\frac{b}{2}\nLtwo{\nabla \op_t(t)}^2\\
\leq&\, \left(c^2\nLtwoLtwo{\Delta \op}+|k| \|q^{(1)}_t\|_{L^\infty(L^3)}\|\op_{t}\|_{L^2(L^6)}+\nLtwoLtwo{f} \right)\nLtwoLtwo{\op_{tt}}.
\end{aligned}
\end{equation}
Combining the estimates derived above and then employing Gronwall's inequality yields
\begin{equation}
\begin{aligned}
&  \nLtwoLtwo{\op_{tt}}^2 + \nLtwo{\nabla \op_t(t)}^2+ \nLtwo{\Delta \op(t)}^2+b \nLtwoLtwo{\D \op_t}^2+b\nLtwo{\nabla \op_t(t)}^2  \\
\leq&\, C_1 \exp{(C_2 T)} \|f\|^2_{L^2 (H^1)},
\end{aligned}
\end{equation}
where the positive constants $C_1$ and $C_2$ do not depend on $b$. From here we have 
\begin{equation} \label{contractivity_est_}
\begin{aligned}
\|\op\|^2_{X}
\leq&\, \begin{multlined}[t] C_3 \exp{(C_2 T)} k^2 (T \| p^{(2)}_t\|^2_{L^\infty (L^\infty)}+ T\|\nabla p^{(2)}_t\|^2_{L^\infty (L^6)}
\\+  \| p^{(2)}_{tt}\|^2_{L^2(L^6)}+\|\nabla p^{(2)}_{tt}\|^2_{L^2(L^2)} )\|\overline{q}\|^2_{X} \end{multlined}\\
\leq&\, C_4 \exp{(C_2 T)} k^2 (T+1) C_{\textup{lin}}(T)(\|p_0\|^2_{H^3}+\|p_1\|^2_{H^2}) \|\overline{q}\|^2_{X},\\
\leq&\, C_5 \exp{(C_6 T)} k^2 (T+1) (\|p_0\|^2_{H^3}+\|p_1\|^2_{H^2}) \|\overline{q}\|^2_{X}.
\end{aligned}
\end{equation}
Alternatively, assuming the initial data to be smoother so that the estimate \eqref{LinWest:Main_energy_est_higher} applies, we can estimate $\| p^{(2)}_{tt}\|^2_{L^2(L^6)}$ by $T\| p^{(2)}_{tt}\|^2_{L^\infty(L^6)}$ and thus get
\begin{equation} 
\|\op\|^2_{X}\leq\, C_5 \exp{(C_6 T)} k^2 \, T\, (\|p_0\|^2_{H^3}+\|p_1\|^2_{H^3}) \|\overline{q}\|^2_{X},
\end{equation}
 in place of \eqref{contractivity_est_}.
Here $C_3, \ldots, C_6>0$ do not depend on $b$. We, therefore, conclude that $\mathcal{T}$ is strictly contractive provided the data is small enough so that
\eqref{small1} or \eqref{small2} holds.
\end{proof}
The two previous results allow us to now employ a contraction principle and obtain a local well-posedness result for the Westervelt equation with a uniform bound in $b$ for small and smooth data.
\begin{theorem} \label{Thm:West_Wellposedness}
Let $\Omega \subset \R^n$, where $n \in \{1, 2, 3\}$, be bounded and $C^3$ regular. Let $k \in \R$. Furthermore, let $b \in [0, \bar{b})$. There exist constants $C_1$, $C_2$, $C_5$, $C_6>0$ that depend only on the domain $\Omega$ and the constant $c^2$, but neither on final time $T$ nor on $b$, such that if, for some $\delta$, $R>0$,  
\begin{equation}
{
C_{\textup{A}}\tilde{C}_1 \exp(\tilde{C}_2RT)\sqrt{\delta} R^{1/2} \leq \frac{1}{2|k|},
}
\end{equation}
as well as  \eqref{small1} or \eqref{small2} hold,
then for any initial data satisfying \eqref{IC_Westervelt}, \eqref{smallnessT}, and 
\begin{equation}
{\nLtwo{ p_1}^2+\nLtwo{\nabla p_0}^2 \leq \delta^2,}
\end{equation}
there exists a unique solution $p\in \spaceW$ of problem 
\begin{equation} \label{ibvp_West}
\left\{
\begin{aligned}
\ptt - b \D p_t - c^2\D p  =&\, \frac{k}{2} (p^2)_{tt}  &&\quad\text{ in }\Omega\times(0,T), \\[1mm]
p=&\,0  &&\quad\text{ on } \partial \Omega\times(0,T),\\[1mm]
(p, p_t)=&\,(p_0, p_1)  &&\quad\mbox{ in }\Omega\times \{0\},
\end{aligned} \right.
\end{equation}
where $\spaceW$ is defined in \eqref{regularity}. Furthermore, solution $p$ satisfies the estimate \eqref{estpnonlin}.
\end{theorem}
\begin{proof}
{The proof follows by employing the Banach fixed-point theorem to the mapping 
	\[
	\mathcal{T}: M \ni q \mapsto p,
	\]
with the ball $M$ defined in \eqref{defM} and $p$ being the solution of 
\begin{equation} 
(1-kq)\ptt - b \D p_t - c^2\D p - k q_t p_t =0,
\end{equation}
with initial conditions $p(0)=q(0)=p_0$, $p_t(0)=q_t(0)=p_1$.}\\
\indent We {first} note that the space $(M, d)$ with the metric $d(x,y)=\|x-y\|_{X}$ is a closed subset of a complete normed space, {where we recall that
	\begin{equation}
\|p\|^2_{X}= \nLtwoLtwo{\ptt}^2+\nLinfLtwo{\nabla \pt}^2 + \nLinfLtwo{\Delta p}^2.
\end{equation}} Indeed, $M$ is a ball in $\spaceW\cap\LinfLinf$, thus it is weakly-$\star$ closed in $\spaceW$ by virtue of the Banach--Alaoglu theorem. Therefore, any Cauchy sequence with respect to the $X$ norm in $M$ converges to some \[y\in L^\infty(0,T; \Honetwo) \cap W^{1, \infty}(0,T; H_0^1(\Omega)) \cap H^2(0,T; L^2(\Omega))\] and has a  weakly-$\star$ in $\spaceW$ convergent subsequence with some limit  $x\in M$. Due to uniqueness of limits we have $y=x\in M$.\\
\indent {Proposition~\ref{Prop:SelfMapping} guarantees the $\mathcal{T}$ is well-defined self-mapping, whereas Proposition~\ref{Prop:Contraction} yields strict contractivity.} The assertion then follows by Banach's fixed-point theorem; cf., e.g., \cite[Theorem 2.1]{Teschl_ode}. Note that this implies uniqueness, first of all, only on $M$. Global uniqueness {in $\spaceW$} follows analogously to the standard uniqueness argument in the Picard-Lindel\"of Theorem under a local Lipschitz condition; see, e.g., \cite[Theorem 1.3]{Teschl_ode}.
\end{proof}
\indent We emphasize that smallness of the initial data is only required in a weaker norm than the one of $\spaceW_0$. The other smallness constraints in Theorem~\ref{Thm:West_Wellposedness} can be achieved by making either $\|(p_0,p_1)\|_{\spaceW_0}$ small or final time $T$ short enough. This theoretical framework thus agrees well with ultrasonic applications, where typically the final time will be short and the data smooth, but not necessarily small; see, for example,~\cite[\S 5]{kaltenbacher2007numerical} and~\cite{walsh2007finite}. \\
\indent When $n=1$, already the space $H^1(\Om)$ embeds into $L^\infty(\Om)$, which allows us to simplify the theoretical framework. Thus we expect that the results of Theorem~\ref{Thm:West_Wellposedness} can be sharpened in that case; however, given the real-world ultrasonic applications that motivate our work, we do not pursue that path here. 

\section{The inviscid limit of the Westervelt equation}  \label{Sec:LimitWest}
Equipped with a uniform bound in $b$, we are now ready to prove a limiting result for the Westervelt equation as $b \rightarrow 0^+$. To formulate the result, we introduce the space
\begin{equation}\label{X0}
\textup{E}=W^{1,\infty}(0,T;L^2(\Omega))\cap L^\infty(0,T;H_0^1(\Omega))
\end{equation}
equipped with the standard energy norm for the wave equation
\begin{equation} \label{energy_norm}
\|p\|^2_{\textup{E}}= \sup_{t\in(0,T)} \nLtwo{p_t(t)}^2+\sup_{t\in(0,T)}\nLtwo{\nabla p(t)}^2. 
\end{equation}
\begin{theorem} \label{Thm:West_WeakLimit}
Under the conditions of Theorem~\ref{Thm:West_Wellposedness}, the family of solutions $\{p^{(b)}\}_{b>0}$ to the Westervelt equation converges in the topology induced by the energy norm for the wave equation to a solution $p$ of the inviscid Westervelt equation at a linear rate
\begin{equation}
\|p^{(b)}-p\|_{\textup{E}}\lesssim b \mbox{ as } b\to0.
\end{equation}
\end{theorem}
\begin{proof}
Following the general approach of~\cite{tani2017mathematical}, the statement follows by proving that $\{p^{(b)}\}$ is a Cauchy sequence in $\textup{E}$. \\
\indent Let $b, b' \in (0, \overline{b})$. Let $p^{(b)}$ and $p^{(b')}$ be the solutions of the Westervelt equation with the sound diffusivity $b$ and $b'$, respectively, and with homogeneous Dirichlet data.  We assume the initial conditions and final time are chosen so that the assumptions of Theorem~\ref{Thm:West_Wellposedness} are satisfied. Since these conditions are independent of $b$, solutions exist on a maximal common interval $[0,T]$ with $T>0$ for all $b, b' \in [0, \overline{b})$.\\
\indent Then $\op=p^{(b)}-p^{(b')}$ solves the equation
\begin{equation} \label{West_Cauchy_eq}
(1-k p^{(b)} )\op_{tt}-c^2 \Delta \op-b \Delta \op_t=k\overline{p}_t (p_t^{(b)}+p^{(b')}_t)+k \overline{p}p^{(b')} _{tt}+(b-b')\Delta p_t^{(b')}.
\end{equation}
Testing with $\op_t$ and integrating over space and time leads to
\begin{equation}
\begin{aligned}
&\frac12\nLtwo{\sqrt{\alpha(t)}\op_t(t)}^2+\frac{c^2}{2}\nLtwo{\nabla \op(t)}^2+b \nLtwoLtwo{\nabla \op_t}^2 \\
=&\,\begin{multlined}[t]-\frac12k\int_0^t \intO p^{(b)}_t \op_t^2 \dxs+  k\int_0^t \intO  (p_t^{(b)}+p_t^{(b')})\op_t^2 \dxs+k\int_0^t \intO p _{tt}^{(b')} \op \op_t \dxs \\
+ (b-b')\int_0^t \intO  \D \pt^{(b')} \op_t \dxs. \end{multlined}
\end{aligned}
\end{equation}
From here, by using Young's inequality and then taking the supremum over $(0,\tau)$, we have the estimate
\begin{equation} \label{est_conv_West}
\begin{aligned}
&\sup_{t\in(0,\tau)} \nLtwo{\op_t(t)}^2+\sup_{t\in(0,\tau)}\nLtwo{\nabla \op(t)}^2 +b \int_0^\tau\nLtwo{\nabla \op_t}^2\ds\\
\lesssim &\,\begin{multlined}[t] (\nLinfLinf{p^{(b)}_t}+\nLinfLinf{p_t^{(b')}} )\int_0^\tau \nLtwo{\op_t}^2 \ds\\
+\varepsilon \sup_{s\in(0,\tau)} \nLtwo{\op_t(s)}^2 
+\nLtwoLfour{p _{tt}^{(b')}}^2 \int_0^\tau\nLfour{\op}^2\ds   
\\+ (b-b')^2 \nLtwoLtwo{\Delta p_t^{(b')}}^2+ \int_0^\tau\nLtwo{\op_t}^2\ds, \end{multlined}
\end{aligned}
\end{equation}
with $\varepsilon>0$ small enough, but independent of $b$.
Dominating the $\varepsilon$ term on the right-hand side by the first term on the left-hand side and using continuous embeddings as well as the uniform bounds $\|p^{(b)}\|_{\spaceW},\, \|p^{(b')}\|_{\spaceW}\, \lesssim R$ resulting from Theorem~\ref{Thm:West_Wellposedness}, we conclude that
\begin{equation} \label{intermediate}
\begin{aligned}
&\sup_{t\in(0,\tau)} \nLtwo{\op_t(t)}^2+\sup_{t\in(0,\tau)}\nLtwo{\nabla \op(t)}^2 +b \int_0^\tau\nLtwo{\nabla \op_t}^2\ds\\
\lesssim &\,\begin{multlined}[t] R\int_0^\tau \nLtwo{\op_t}^2 \ds+ R\int_0^\tau\nLtwo{\nabla\op}^2\ds   
+ R^2(b-b')^2  \end{multlined}
\end{aligned}
\end{equation}
since $p^{(b)}$, $p^{(b')}\in M$. Gronwall's inequality therefore yields
\begin{equation}
\|p^{(b)}-p^{(b')}\|_{\textup{E}}=\|\op\|_{\textup{E}}\lesssim |b-b'|.
\end{equation}
This estimate, which remains valid for $b'=0$ also yields the convergence rate 
\begin{equation}
\|p^{(b)}-p\|_{\textup{E}}=\|\op\|_{\textup{E}}\lesssim b,
\end{equation}
where $p$ is a solution to the inviscid version of the Westervelt equation.
\end{proof}
\begin{remark}[Convergence in a higher-order norm] \label{rem:rate}
Similarly to the proof of contractivity in Proposition~\ref{Prop:Contraction}, we can also establish convergence with respect to the topology used there
\begin{equation}\label{X}
X=H^2(0,T;L^2(\Omega))\cap W^{1,\infty}(0,T;H_0^1(\Omega))\cap L^\infty(0,T;\Honetwo);
\end{equation}
that is, a stronger topology than $\textup{E}$, with a slower convergence rate 
\begin{equation}\label{rateX}
\|p^{(b)}-p\|_{X}\lesssim \sqrt{b} \mbox{ as } b\to0.
\end{equation}
Indeed, estimating the right-hand side terms up to the last one in \eqref{West_Cauchy_eq} goes through analogously to the estimates for the difference equation \eqref{West_contract_eq}; cf. \eqref{estf}. For the last term we obtain 
\begin{equation}
\nLtwoLtwo{(b-b')\Delta p_t^{(b')}} \leq |b-b'|\, b^{-1/2}R,
\end{equation}
with $R$ as in \eqref{defM}; see also \eqref{smallnessinit} and \eqref{estpnonlin}. As in the proof of Proposition~\ref{Prop:Contraction}, this yields 
\[\|p^{(b)}-p\|_{X}\lesssim |b-b'| b^{-1/2}.\] 
Thus, with $b'=0$ above, we obtain the convergence rate \eqref{rateX}.
\end{remark} 

\section{Analysis of the Kuznetsov equation}  \label{Sec:Kuznetsov}
We wish to extend our considerations of the vanishing sound diffusivity dynamics to the Kuznetsov equation next. As a by-product, we will also obtain results for the Westervelt equation in the acoustic potential formulation \eqref{West_potential} by setting $\sigma=0$ and $\kappa =\frac{2}{c^{2}}(\frac{B}{2A}+1)$ in \eqref{Kuznt}. \\
\indent Compared to the analysis of the Westervelt equation, now the factor next to the second time derivative involves $\psi_t$. This implies that a uniform bound for $\|\psi_t\|_{L^\infty(L^\infty)}$ is needed. Together with the presence of quadratic gradient nonlinearity, higher-order energy analysis is required in this case to arrive in the end at a uniform bound with respect to the parameter $b$.  
\subsection{Energy analysis of the linearized Kuznetsov equation}
Similarly to before, we begin by studying a linear non-degenerate version of the nonlinear equation in question. We consider the following linearization:
\begin{equation} \label{ibvp_linKuzn}
\left\{
\begin{aligned}
\alpha\psi_{tt} -c^2\D \psi - b \D \psi_t=&\, \sigma\nabla\phi\cdot\nabla\psi_t  &\quad\text{ in }\Omega\times(0,T), \\[1mm]
\psi=&\,0  &\quad\text{ on } \partial \Omega\times(0,T),\\[1mm]
(\psi, \psi_t)=&\,(\psi_0, \psi_1)  &\quad\mbox{ in }\Omega\times \{0\},
\end{aligned} \right.
\end{equation}
with $\alpha=1-\kappa \phi_t$ and $\kappa \in \R$.  By abbreviating $p=\psi_t$, $q=\phi_t$, the PDE above can be rewritten as 
\begin{equation}\label{Kuznetsov_psi}
\alpha p_t - b \D p -c^2\D \psi = \sigma\nabla\phi\cdot\nabla p 
\end{equation}
with $\alpha=1-\kappa q $. \\
\indent The functions $\phi$ and $q=\phi_t$ are required to have the following regularity:
\begin{equation}\label{reg_qphi_linKuz}
\phi\in L^\infty(0,T; \Honefour), \ q\in L^\infty(0,T;\Honethree), \ q_t\in L^\infty(0,T;\Honetwo),
\end{equation}
which implies that the coefficient $\alpha$ has the regularity 
\begin{equation}\label{eq:alphagammaf_reg_Kuzn}
\begin{aligned} 
&\alpha \in \spaceK_{\alpha}=
L^\infty(0,T;\Honethree)\cap W^{1,\infty}(0,T;\Honetwo);
\end{aligned}
\end{equation}
cf. \eqref{sobolev_withtraces}. 
Again, a compatibility condition on the initial data will be needed, which here reads as 
\begin{equation}\label{compat_Kuz}
\psi_{tt}(0)=\psi_2 =\alpha(0)^{-1} (b \D \psi_1 +c^2\D \psi_0 + \sigma\nabla\phi(0)\cdot\nabla\psi_1)\,.
\end{equation}
We are now ready to state a well-posedness result for this linear problem.
\begin{proposition} \label{Prop:LinKuzn}
	Assume that the domain $\Omega \subset \R^n$, where $n \in \{1, 2, 3\}$, is bounded and $C^4$ regular. Let $b \in [0, \bar{b})$ and  $T>0$. Let $\phi\in L^\infty(0,T;\Honefour)$. Assume that  $\alpha \in \spaceK_{\alpha}$ and that there exist $\ulal$, $\olal>0$ such that
	\begin{equation} \label{nondegeneracy_Kuzn}
	\ulal \leq \alpha(x, t)\leq \olal \ \mbox{ on }\Omega \ \ \text{  a.e. in } \Om \times (0,T).
	\end{equation}
	Furthermore, assume that
	\begin{equation} \label{IC_Kuznetsov}
	(\psi_0, \psi_1) \in \spaceK_0=\Honefour\times \Honethree
	\end{equation}
and \eqref{compat_Kuz} holds.
	 Then there exists a unique  solution $\psi$ of the problem \eqref{ibvp_linWest} such that
	\begin{equation}\label{regularity_Kuzn}
	\begin{aligned}
	\psi \in \, \spaceK =& \,\begin{multlined}[t] H^3(0,T;H_0^1(\Omega))\cap W^{2,\infty}(0,T;\Honetwo)\\ 
												\cap W^{1,\infty}(0,T;\Honethree)\cap L^\infty(0,T;\Honefour).
	\end{multlined}
	\end{aligned}
	\end{equation}
	Furthermore, this solution satisfies the estimate
\begin{equation} \label{LinKuzn:Main_energy_est}
\begin{aligned}
& \begin{multlined}[t] \nLtwoLtwo{\nabla \psi_{ttt}}^2\,+\esssup_{t \in (0,T)} \nLtwo{\D \psi_{tt} (t)}^2\,
+\esssup_{t \in (0,T)} \nLtwo{\nabla \D \psi_t(t)}^2 \\+\esssup_{t \in (0,T)} \nLtwo{\D^2 \psi(t)}^2 
+b\nLtwoLtwo{\nabla \D \psi_{tt}}^2 +b\nLtwoLtwo{\D^2 \psi_t}^2
\end{multlined}\\
\lesssim&\, \begin{multlined}[t] \nLtwo{\D^2 \psi_0}^2+\nLtwo{\nabla\D \psi_1}^2+\nLtwo{\D[\nabla\phi(0)\cdot\nabla\psi_1]}^2,
\end{multlined}
\end{aligned}
\end{equation}
where \[\nLtwo{\D[\nabla\phi(0)\cdot\nabla\psi_1]}\leq 2(\CHtwo+\CHone^2)\nLtwo{\nabla\D\phi(0)}\nLtwo{\nabla\D \psi_1}.\]  
The hidden constant in \eqref{LinKuzn:Main_energy_est} tends to infinity as $T \rightarrow \infty$, but does not depend on the parameter $b$.
\end{proposition}
\begin{proof}
We start by proving existence of solutions to the time-differentiated version of our original problem \eqref{ibvp_linKuzn}. We thus consider
\begin{equation} \label{ibvp_linKuzn_timediff}
\left\{
\begin{aligned}
\alpha p_{tt} - b \D p_t -c^2\D p =&-\alpha_t p_t + \sigma\nabla q\cdot\nabla p+\sigma\nabla\phi\cdot\nabla p_t \hspace*{-4mm}  &&\ \text{ in }\Omega\times(0,T), \\[1mm]
p=&\,0  &&\ \text{ on } \partial \Omega\times(0,T),\\[1mm]
(p, p_t)=&\,(\psi_1, \psi_2)  &&\ \mbox{ in }\Omega\times \{0\},
\end{aligned} \right.
\end{equation}
where the function $\psi_2$ is defined by \eqref{compat_Kuz}.

The proof can be carried out as before via smooth Faedo--Galerkin approximations in space, by projecting the problem \eqref{ibvp_linKuzn_timediff} onto the span $V_n$ of the first $n$ eigenfunctions of the Laplacian pointwise in time. \\
\indent We focus here on deriving the crucial uniform bound for the approximate solution since the remaining steps are analogous to the proof of Proposition~\ref{Prop:LinWest}. For simplicity of notation, we drop the superscript $n$ in the approximate solution.
\\
\indent In the first step, we multiply \eqref{ibvp_linKuzn_timediff} by $\D^2 \pt$ and integrate over $\Om$:
\begin{equation}\label{Kuznetsov_p}
\prodLtwo{\alpha p_{tt} - b \D p_t -c^2\D p}{\D^2 \pt}=\prodLtwo{-\alpha_t p_t + \sigma\nabla q\cdot\nabla p + \sigma\nabla\phi\cdot\nabla p_t}{\D^2 \pt}. 
\end{equation}
 When deriving the energy bound, we can rely on the identity
\[
\begin{aligned}
\prodLtwo{\alpha p_{tt}}{\D^2 p_t} 
=&\,\, \prodLtwo{\D[\alpha p_{tt}]}{\D p_t} \\
=&\, \frac{\textup{d}}{\textup{d}t}\frac12 \prodLtwo{\alpha \D p_t}{\D p_t} +\frac12 \kappa \prodLtwo{q_t \D p_t}{\D p_t}
-\prodLtwo{\kappa \D q\, p_{tt} +\kappa\nabla q\cdot\nabla p_{tt}}{\D p_t}\\
=&\, \frac{\textup{d}}{\textup{d}t}\frac12 \prodLtwo{\alpha \D p_t}{\D p_t} + \prodLtwo{f_1(\phi,q,p)}{\D p_t}, 
\end{aligned}
\]
where we have used the fact that $\D p$ vanishes on the boundary for smooth approximations of the problem and introduced the short-hand notation
\begin{equation}
\begin{aligned}
f_1(\phi,q,p) = \frac12 \kappa q_t \D \pt-\kappa \D q \ptt - \kappa\nabla q \cdot \nabla \ptt.
\end{aligned}
\end{equation}
Recalling that $\ptt=\D\pt=\D p=0$ on $\partial \Om$, which via the PDE implies that also $\nabla q \cdot \nabla p+\nabla \phi \cdot \nabla \pt$ vanishes on the boundary, we further find that
\begin{equation}
\begin{aligned}
&\prodLtwo{-\alpha_t p_t + \sigma\nabla q\cdot\nabla p + \sigma\nabla\phi\cdot\nabla p_t}{\D^2 \pt}\\
=&\,\begin{multlined}[t]\prodLtwo{\D[-\alpha_t p_t]}{\D \pt}+\sigma\prodLtwo{\D[\nabla q\cdot\nabla p + \nabla\phi\cdot\nabla p_t]}{\D \pt}.
\end{multlined}
\end{aligned}
\end{equation}
Furthermore, we have
\[
\begin{aligned}
&\D[-\alpha_t p_t + \sigma\nabla q\cdot\nabla p + \sigma\nabla\phi\cdot\nabla p_t]\\
=&\,\begin{multlined}[t] \kappa \Delta q_t \, p_t + 2\kappa\nabla q_t\cdot\nabla p_t + \kappa q_t\,\D p_t
+ \sigma\nabla \D q\cdot\nabla p+ 2\sigma D^2 q:D^2 p\\+ \sigma\nabla q\cdot\nabla \D p
+ \sigma\nabla \D \phi\cdot\nabla p_t+ 2\sigma D^2 \phi:D^2 p_t+ \sigma\nabla \phi\cdot\nabla \D p_t \end{multlined}\\
=&\, f_2(\phi,q,p)+\sigma\nabla \phi\cdot\nabla \D p_t, 
\end{aligned}
\]
where we have introduced another short-hand notation
\begin{equation}
\begin{aligned}
f_2(\phi,q,p)= \kappa \Delta q_t \, p_t + 2\kappa\nabla q_t\cdot\nabla p_t + \kappa q_t\,\D p_t
+ \sigma\nabla \D q\cdot\nabla p+ 2\sigma D^2 q:D^2 p\\+ \sigma\nabla q\cdot\nabla \D p
+ \sigma\nabla \D \phi\cdot\nabla p_t+ 2\sigma D^2 \phi:D^2 p_t\,.
\end{aligned}
\end{equation}
By elliptic regularity the Hessian $D^2 v=(\partial_{x_i} \partial_{x_j} v)_{i,j}$ satisfies 
{
\[
\|D^2 v\|_{L^p}\leq C_{\textup{H}} \|\Delta v\|_{L^p} 
\quad \mbox{ for all }v\in H_0^1(\Omega)\cap W^{2,p}(\Omega)\, \quad p\in (1,6],
\] 
see, e.g., \cite[Theorem 2.4.2.5]{Grisvard}.
}
Due to the PDE, we know that
\begin{equation}\label{pttKuz}
\begin{aligned}
p_{tt}=&\,\ooal\left(b \D p_t +c^2\D p + \kappa q_t p_t + \sigma\nabla q\cdot\nabla p + \sigma\nabla\phi\cdot\nabla p_t\right),\\
\nabla p_{tt}=&\, \begin{multlined}[t]\kappa\alpha^{-2}\nabla q\,\left(b \D p_t +c^2\D p + \kappa q_t p_t + \sigma\nabla q\cdot\nabla p + \sigma\nabla\phi\cdot\nabla p_t\right)\\
 +\ooal \left(b \nabla\D p_t +c^2\nabla\D p + \kappa \nabla q_t p_t+ \kappa q_t \nabla p_t
+ \sigma D^2 q\,\nabla p \right. \\ \left.+ \sigma D^2 p\,\nabla q + \sigma D^2\phi\,\nabla p_t + \sigma D^2 p_t\,\nabla\phi
\right). \end{multlined}
\end{aligned}
\end{equation}
Altogether, we have the energy identity 
\[
\begin{aligned}
&\frac12\ddt\nLtwo{\sqrt{\alpha}\D p_t}^2 + b \nLtwo{\nabla \Delta p_t}^2 + c^2 \frac12\frac{\textup{d}}{\textup{d}t} \nLtwo{\nabla \D p}^2 \\[1mm]
&= \begin{multlined}[t]\prodLtwo{-f_1(\phi,q,p)+f_2(\phi,q,p)+\sigma\nabla \phi\cdot\nabla \D p_t}{\D p_t}.
\end{multlined}
\end{aligned}
\]
We can estimate the right-hand side by employing the following bound:
\[
\begin{aligned}
\nLtwo{f_1(\phi,q,p)}\leq&\, \begin{multlined}[t] \frac12|\kappa| \CHtwo \nLtwo{\D q_t}\nLtwo{\D p_t}
\\+  (|{\kappa}|\CHtwo+|\kappa|\CHone^2) \nLtwo{\nabla \D q} \nLtwo{\nabla p_{tt}}. \end{multlined}
\end{aligned}
\]
Furthermore, by virtue of identities \eqref{pttKuz}, we have
\[
\begin{aligned}
\nLtwo{\nabla p_{tt}}\leq&\, \begin{multlined}[t] 
\frac{|\kappa|}{\ulal^2} \CHtwo \nLtwo{\nabla \D q} 
\left(b\nLtwo{\D p_t}+c^2\nLtwo{\D p}
+\CHone^2 \left(|\kappa|\nLtwo{\nabla q_t}\nLtwo{\nabla p_t} \right. \right.\\ \left. \left.
+ \abssig \nLtwo{\D q}\nLtwo{\D p} + \abssig \nLtwo{\D \phi}\nLtwo{\D p_t}\right) \vphantom{\CHtwo}\right)\\
+\frac{1}{\ulal} \left(b \nLtwo{\nabla \D p_t}+c^2\nLtwo{\nabla \D  p}
+2|\kappa| \CHone \CPF^{1/2} \nLtwo{\D q_t}\nLtwo{\D p_t}\right. \\ \left.
+\abssig C_{\textup{H}} (\CHtwo+\CHone^2) (\nLtwo{\nabla \D q}\nLtwo{\D p}+\nLtwo{\nabla \D \phi}\nLtwo{\D p_t})
\right). \end{multlined}
\end{aligned}
\]
We note that also
\[
\begin{aligned}
\nLtwo{f_2(\phi,q,p)}\leq&\, \begin{multlined}[t] 
2|\kappa|(\CHtwo+\CHone^2) \nLtwo{\D q_t}\nLtwo{\D p_t} \\
+ 2\abssig (\CHtwo+C_{\textup{H}}^2 \CHone ^2) \nLtwo{\nabla \D q}\nLtwo{\nabla \D p} \\
+ \abssig (C_{\textup{H}} \CHtwo+2 \CHone^2) \nLtwo{\D^2\phi}\nLtwo{\D p_t} \end{multlined}
\end{aligned}
\]
and
\[
\begin{aligned}
2 \prodLtwo{\nabla \phi\cdot\nabla \D p_t}{\D p_t} = \prodLtwo{\nabla \phi}{\nabla (\D p_t)^2} =&\, \prodLtwo{\D \phi}{(\D p_t)^2}\\
\leq&\, \CHtwo  \nLtwo{\Delta^2\phi}\nLtwo{\D p_t}^2. 
\end{aligned}
\]
Under the regularity assumptions \eqref{reg_qphi_linKuz}, similarly to the proof of Proposition \ref{Prop:LinWest}, by using Young's and Gronwall's inequalities, we arrive at an energy estimate of the form
\begin{equation} \label{LinKuzn:Lower_energy_est}
\begin{aligned}
& \begin{multlined}[t] \nLtwoLtwo{\nabla \ptt}^2\,+\sup_{t \in (0,T)} \nLtwo{\D \pt (t)}^2\,+\sup_{t \in (0,T)} \nLtwo{\nabla \D p(t)}^2 
+b\nLtwoLtwo{\nabla \D \pt}^2
\end{multlined}\\
\lesssim&\, \begin{multlined}[t] \nLtwo{\nabla \D \psi_1}^2+\nLtwo{\D \psi_2}^2,
\end{multlined}
\end{aligned}
\end{equation}
first of all, for the Galerkin approximations. Via weak limits we obtain existence, uniqueness, and the energy estimate \eqref{LinKuzn:Lower_energy_est} also for a solution to \eqref{ibvp_linKuzn_timediff}.
Integrating with respect to time and using compatibility of the initial data \eqref{compat_Kuz} yields a unique solution to \eqref{ibvp_linKuzn}.\\
\paragraph{\bf Additional regularity}
The terms containing $\nLtwo{\Delta^2\phi}\nLtwo{\D p_t}$ above and in $f_2$ show that another energy estimate is needed to obtain an overall bound that is independent of $b$ for the nonlinear problem. To this end, we exploit elliptic regularity and the fact that by the PDE, $\D\psi$ satisfies
\begin{equation}\label{Deltapsi}
(b\partial_t+c^2)\D \psi = r \mbox{ with } r=\alpha p_t -\sigma \nabla\phi\cdot\nabla p.
\end{equation}
Thus in case $b>0$, we have
\begin{equation}\label{Deltapsioft}
\D \psi (t) = g_0(t) \D \psi(0) + \int_0^t g'(t-s) r(s)\ds
\end{equation}
with \[g(t)=-c^{-2} \exp{\left(-\frac{c^2}{b} t\right)},\qquad g_0(t)=\exp{\left(-\frac{c^2}{b} t\right)}\in[0,1].\] 
By positivity of $g'$, we have 
\[\|g'\|_{L^1(0,t)}=\int_0^t g'(s)\ds = g(t)-g(0) = c^{-2}\left(1-\exp{\left(-\frac{c^2}{b} t\right)}\right)\leq c^2.\]
Therefore,
\begin{equation}\label{L2Deltapsi}
\nLtwo{\D^2 \psi (t)} 
\leq \nLtwo{\D^2 \psi(0)} + c^{-2} \|\D r\|_{L^\infty(0,t;L^2(\Omega)}.
\end{equation}
In case $b=0$, estimate \eqref{L2Deltapsi} (without the $\nLtwo{\D^2 \psi(0)}$ term) immediately follows from \eqref{Deltapsi}. The right hand-side norm can be estimated as follows:
\begin{equation}
\begin{aligned}
&\|\D r\|_{L^\infty(L^2)}\\
=&\, \|\kappa \D q \, p_t + 2\kappa\nabla q\cdot\nabla p_t + \alpha\,\D p_t
+ \sigma\nabla \D \phi\cdot\nabla p+ 2\sigma D^2 \phi:D^2 p+ \sigma\nabla \phi\cdot\nabla \D p
\|_{L^\infty(L^2)}\\
\leq&\, \begin{multlined}[t]{(1+|\kappa| (2C_{H^2, L^\infty}+2C_{\textup{H}}^2 C_{H^1, L^4}^2)}
\|\D q\|_{L^\infty(L^2)} )\|\D p_t\|_{L^\infty(L^2)} \\
+2 |\sigma| ({C_{\textup{H}}}
C_{H^2, L^\infty}+C_{\textup{H}}^2 C_{H^1, L^4}^2)
\|\nabla \D \phi\|_{L^\infty(L^2)} \|\nabla \D p\|_{L^\infty(L^2)}\,,\end{multlined}
\end{aligned}
\end{equation}
and therefore bounded by means of the already established estimate \eqref{LinKuzn:Lower_energy_est}.
{
Here we have used the elliptic regularity result (see, e.g., \cite[Lemma 1, page 34]{thomee2006galerkin})
\[
\|v\|_{H^3}\leq C_{\textup{H}} \|(-\Delta)^{3/2} v\|_{L^2} 
\quad \mbox{ for all }v\in H_0^1(\Omega)\cap H^3(\Omega),
\] 
and estimated 
\[
\|\nabla p\|_{L^\infty}
\leq C_{H^2, L^\infty} \|p\|_{H^3}
\leq C_{H^2, L^\infty} C_{\textup{H}} \|(-\Delta)^{3/2} p\|_{L^2}
= C_{H^2, L^\infty} C_{\textup{H}} \|\nabla\Delta p\|_{L^2}\,,
\]
and in the same manner bounded $\|\nabla \phi\|_{L^\infty}$, due to the fact that $\Delta p$ and $\Delta \phi$ vanish on $\partial\Omega$.
}
Furthermore, by \eqref{Deltapsioft}, the function $\psi$ satisfies the boundary conditions 
\begin{equation}
\begin{aligned}
&\psi\vert_{\partial\Omega}=0\,, \\
&
\D\psi\vert_{\partial\Omega}={\mathrm{bdy}}:=\begin{cases}
g_0(t) \D \psi\vert_{\partial\Omega}(0) + \int_0^t g'(t-s)(\sigma \nabla\phi\cdot\nabla p)\vert_{\partial\Omega}(s)\ds&\mbox{ if }b>0\\
(\sigma \nabla\phi\cdot\nabla p)\vert_{\partial\Omega}&\mbox{ if }b=0\,,
\end{cases}
\end{aligned}
\end{equation}
where 
\begin{equation}
\begin{aligned}
\|(\sigma \nabla\phi\cdot\nabla p)\vert_{\partial\Omega}\|_{L^\infty(0,T;H^{3/2}(\partial\Omega))}
\leq C_{\textup{tr}} \|\sigma \nabla\phi\cdot\nabla p\|_{L^\infty(0,T;{H^2(\Omega)})}
\end{aligned}
\end{equation}
can again be estimated by means of \eqref{LinKuzn:Lower_energy_est}.
\\
\indent Thus we can invoke higher elliptic regularity \cite[Theorem 8.14]{salsa2015pde} to first conclude 
\[
\|\psi\|_{L^\infty(0,T;H^4(\Omega))}\leq C_{\textup{H}}\Bigl( \|\D\psi\|_{L^\infty(0,T;H^2(\Omega)}+\|\psi\|_{L^\infty(L^2)}\Bigr)
\]
and then further estimate  
\[
\|\D\psi\|_{L^\infty(0,T;H^2(\Omega))}\leq C_{\textup{H}}\Bigl( \|\D r\|_{L^\infty(L^2)} + {\|\mathrm{bdy}\|_{L^\infty(H^{3/2}(\partial \Omega))}} + \|\psi\|_{L^\infty(L^2)}\Bigr).
\]

Altogether, by taking into account the derived bounds, we obtain the energy estimate \eqref{LinKuzn:Main_energy_est}; cf. \eqref{LinKuzn:Lower_energy_est}, where we estimate $\nLtwo{\D \psi_2}^2$ by means of \eqref{compat_Kuz}.
\end{proof}
\subsection{Uniform bounds for the Kuznetsov equation}
We next employ a fixed-point argument to analyze the Kuznetsov equation for small data as well as obtain a uniform in $b$ energy bound for its solution. To this end, we introduce the mapping
\begin{equation}
\TK:\phi \mapsto \psi,
\end{equation}
where, similarly to before, $\phi$ will belong to a suitably chosen ball in the space $\spaceK$, and $\psi$ will solve the linearized Kuznetsov equation
\begin{equation} \label{LinKuznetsov}
\alpha\psi_{tt} -c^2\D \psi - b \D \psi_t =\, \sigma\nabla\phi\cdot\nabla\psi_t
\end{equation}
with $\alpha=1-\kappa \phi_t$ and initial data $(\psi(0), \psi_t(0))=(\phi(0), \phi_t(0))=(\psi_0, \psi_1)$. We note that the mapping $\TK$ is well-defined on account of Proposition~\ref{Prop:LinKuzn}, whose assumptions we verify below.
\begin{theorem}\label{Thm:WellpKuzn}
Let $\Omega \subset \R^n$, where $n \in \{1, 2, 3\}$, be bounded and $C^4$ regular. Let $\kappa$, $\sigma \in \R$. Furthermore, let $b \in [0, \bar{b})$ and let the initial data $(\psi_0, \psi_1) $ satisfy the regularity assumption \eqref{IC_Kuznetsov} with
\begin{equation}
\|\psi_0\|^2_{H^4}+\|\psi_1\|^2_{H^3} \leq \delta^2.
\end{equation}
 There exist $\tilde{\delta}>0$ and a final time $T=T(\delta)>0$, independent of $b$, such that if 
\begin{equation}
{\|\psi_0\|^2_{H^2}+\|\psi_1\|^2_{H^2} \leq \tilde{\delta}^2,}
\end{equation}
then there exists a unique solution $\psi \in \spaceK$ of the problem 
\begin{equation} \label{ibvp_West}
\left\{
\begin{aligned}
\psi_{tt}- b \D \psi_t - c^2\D \psi  =&\, \frac12 (\kappa \psi_t^2 + \sigma |\nabla \psi|^2)_{t}  &&\quad\text{ in }\Omega\times(0,T), \\[1mm]
\psi=&\,0  &&\quad\text{ on } \partial \Omega\times(0,T),\\[1mm]
(\psi, \psi_t)=&\,(\psi_0, \psi_1)  &&\quad\mbox{ in }\Omega\times \{0\},
\end{aligned} \right.
\end{equation}
where $\spaceK$ is defined in \eqref{regularity_Kuzn}. Furthermore, the solution $\psi$ satisfies the estimate \eqref{LinKuzn:Main_energy_est}.
\end{theorem}
\begin{proof}
We begin by proving that the mapping $\TK:\phi \mapsto \psi$ is well-defined and a self-mapping on
	\begin{equation}\label{defM_Kuzn}
	\begin{aligned}
	M_{\textup{K}} = \left\{ \vphantom{\sup_{t \in (0,T)} \nLtwo{\D^2 \psi(t)}^2 
		\leq \RK^2 }\right. \phi \in \spaceK \, : \, &\phi(0)= \psi_0, \ \phi_t(0)=\psi_1,\\
&\ \left. \vphantom{\spaceK}\|\phi_t\|_{L^\infty(L^\infty)} \leq \frac{1}{2 |\kappa|}, \right. 
\ \left.\|\phi\|_{\spaceK} \leq \RK \  \vphantom{\sup_{t \in (0,T)} \nLtwo{\D^2 \psi(t)}^2 
	\leq \RK^2 } \right\},
	\end{aligned}
	\end{equation}
where the radius $R_{\textup{K}}$ will be specified below. We first note that the conditions of Proposition~\ref{Prop:LinKuzn} are satisfied. Indeed, regularity assumptions \eqref{reg_qphi_linKuz} and \eqref{IC_Kuznetsov} hold with the uniform bounds
\begin{equation}
\begin{aligned}
\|\phi\|_{L^\infty(H_\diamondsuit^4)} \leq\, \RK, \qquad
\|\alpha\|_{\spaceK_{\alpha}} \lesssim\, 1+ \RK.
\end{aligned}
\end{equation}
Furthermore, we have
\[
\begin{aligned}
\nLinfLinf{1-\alpha}\leq |\kappa|\nLinfLinf{\phi_t} \leq \frac12.
\end{aligned}
\]
We can thus take $\ulal=\frac12$ and $\olal=\frac32$ in \eqref{nondegeneracy_Kuzn}, and so the non-degeneracy condition is fulfilled as well. Then Proposition~\ref{Prop:LinKuzn} applied to $\psi=\TK(\phi)$ implies that the mapping is well-defined and, moreover,
 \begin{equation} 
 \begin{aligned}
  \|\psi\|^2_{\spaceK}  \leq\,  C_{\textup{lin}, \textup{K}}(T, \RK)(\nLtwo{\D^2 \psi_0}^2+\nLtwo{\nabla\D \psi_1}^2+\nLtwo{\D \nabla\psi_0}^2),
 \end{aligned}
 \end{equation}
 where $C_{\textup{lin}, \textup{K}}(T, \RK)$ is determined by the hidden constant in the linear bound \eqref{LinKuzn:Main_energy_est}
{ 
and by inspection of \eqref{LinKuzn:Lower_energy_est} can be seen to have a form similar to the Westervelt case \eqref{Clin_West}
\begin{equation}
C_{\textup{lin}}(T,R)= C_1 \exp{\left(C_2(|\kappa|+|\sigma|)(\RK+\RK^2+\RK^3)T\right)}, 
\end{equation}  
with constants $C_1$ and $C_2$ independent of $\RK$, $T$, and $b \in [0, \bar{b})$ for fixed $\bar{b}$. 
}
Therefore, provided that $\RK$ and final time $T$ are chosen so that \[C_{\textup{lin}, \textup{K}}(T,\RK) (\nLtwo{\D^2 \psi_0}^2+\nLtwo{\nabla\D \psi_1}^2+\nLtwo{\D \nabla\psi_0}^2) \leq \RK^2,\] the solution $\psi$ satisfies the $\RK^2$ bound in \eqref{defM_Kuzn}.\\
\indent To obtain the $1/{(2|\kappa|)}$ bound, we will, similarly to the Westervelt case, derive a uniform in $b$ estimate of $\|\psi_t \|_{L^\infty(H^1)}$ and then rely on Agmon's interpolation inequality. By multiplying the time-differentiated equation \eqref{LinKuznetsov} with $p_{t}$ and integrating over space and time, we arive at	
	\begin{equation} \label{first_est_selfm_Kuzn}
\begin{aligned}
&\frac12 \left\{\|\sqrt{\alpha(t)}p_t(t)\|^2_{L^2}+c^2\|\nabla p(t)\|^2_{L^2} \right\}\Big \vert_0^t+b \int_0^t\|\nabla p_t\|^2\dxs\\
=&-\frac12 \int_0^t(\alpha_t p_t, p_t)_{L^2}\ds+\sigma\int_0^t (\nabla q \cdot \nabla p+\nabla \phi \cdot \nabla p_t, p_t)\ds\\
\lesssim &\,\begin{multlined}[t] \|\alpha_t\|_{L^\infty(L^\infty)}\|p_t\|^2_{L^2(0,t;L^2)}+\|\nabla q\|_{L^\infty(L^\infty)}\|\nabla p\|_{L^2(0, t;L^2)}\|p_t\|_{L^2(0, t;L^2)}\\+\sigma\int_0^t (\nabla \phi \cdot \nabla p_t, p_t)\ds.   \end{multlined}
\end{aligned}
\end{equation}
To estimate the last term on the right, we can first rely on the identity
\begin{equation} \label{self_m_1}
\begin{aligned}
& \int_0^t \int_{\Omega} (\nabla \phi \cdot \nabla p_t) p_t \dxs \\
=&\, \int_0^t \int_{\Omega} (- \D \phi)  p_t p_t \dxs - \int_0^t \int_{\Omega} p_t (\nabla \phi \cdot \nabla p_t)  \dxs 
\end{aligned}
\end{equation}
since $p_t=0$ on $\partial \Omega$, which then leads to the bound
\begin{equation}  \label{self_m_2}
\begin{aligned}
\left|2\int_0^t \int_{\Omega} (\nabla \phi \cdot \nabla \opsi_t) p_t \dxs \right| 
=&\, \left | \int_0^t \int_{\Omega} (- \D \phi)  \opsi_t p_t \dxs \right|\\
\leq&\,  \|\D \phi\|_{L^\infty(0,t; L^\infty)}\|p_t\|_{L^2(0,t; L^2)}^2.
\end{aligned}
\end{equation}
Thus, utilizing the above estimate in \eqref{first_est_selfm_Kuzn}	and employing Gronwall's inequality leads to
\begin{equation} 
\begin{aligned}
&\|p_t(t)\|^2_{L^2}+\|\nabla p(t)\|^2_{L^2} \\
\leq&\, \vardbtilde{C}_1\exp(\vardbtilde{C}_2 \RK T)(\|\psi_2\|^2_{L^2}+\|\nabla \psi_1\|^2_{L^2} )\\
\leq&\,\vardbtilde{C}_1\exp(\vardbtilde{C}_2 \RK T)(\|\alpha(0)^{-1} (b \D \psi_1 +c^2\D \psi_0 + \sigma\nabla\psi_0\cdot\nabla\psi_1)\|^2_{L^2}+\|\nabla \psi_1\|^2_{L^2} ),
\end{aligned}
\end{equation}
where we have also relied on the compatibility of initial data \eqref{compat_Kuz} in the last line. From here, noting that
\[\|\nabla\psi_0\cdot\nabla\psi_1\|_{L^2} \leq \|\nabla \psi_0\|_{L^4}\|\nabla \psi_1\|_{L^4} \lesssim \|\nabla \psi_0\|_{H^1}\|\nabla \psi_1\|_{H^1},\] we conclude that there exist positive constants $\vardbtilde{C}_2$ and $\vardbtilde{C}_3$, independent of $b$, such that
\begin{equation} 
\begin{aligned}
\|p_t(t)\|^2_{L^2}+\|\nabla p(t)\|^2_{L^2} 
\leq\,\vardbtilde{C}_3\exp(\vardbtilde{C}_2 \RK T)(\|\D \psi_1\|^2_{L^2} +\|\D \psi_0\|^2_{L^2}+\|\nabla \psi_1\|^2_{L^2} )
\end{aligned}
\end{equation}
for all $t \in [0,T]$. The non-degeneracy bound follows from the above estimate and Agmon's inequality:
\[
\|p\|_{L^\infty(L^\infty)} \leq C_{\textup{A}}\vardbtilde{C}_4\exp(\vardbtilde{C}_5 RT)(\|\D \psi_1\|^2_{L^2} +\|\D \psi_0\|^2_{L^2}+\|\nabla \psi_1\|^2_{L^2} )^{1/4}\RK^{1/2},
\]
for some positive constants $\vardbtilde{C}_4$, $\vardbtilde{C}_5$ independent of $b$, if we choose $\tilde{\delta}>0$ small enough, so that
\[
C_{\textup{A}}\vardbtilde{C}_4\exp(\vardbtilde{C}_5 RT)\sqrt{\tilde{\delta}}\,\RK^{1/2} \leq \frac{1}{2 |\kappa|}.
\]
Therefore,
 \[\TK(\MK) \subset \MK.\]  
\indent We next prove strict contractivity of the mapping $\TK$ in the energy norm for sufficiently small data. To this end, we take $\phi^{(1)}$ and $\phi^{(2)}$ in $\MK$, and denote their difference by $\overline{\phi}= \phi^{(1)} -\phi^{(2)} $. Let then $\psi^{(1)}=\TK (\phi^{(1)})$ and $\psi^{(2)}=\TK (\phi^{(2)})$. The difference $\opsi=\psi^{(1)} -\psi^{(2)} \in  \spaceK$ solves the equation
\begin{equation} \label{Kuzn_contractivity_eq}
\begin{aligned}
(1-\kappa \phi_t^{(1)} )\opsi_{tt}-c^2 \Delta \opsi-b \Delta \opsi_t
= \kappa \overline{\phi}_t \psi^{(2)}_{tt}+\sigma \nabla \overline{\phi} \cdot \nabla \psi^{(1)}_t+\sigma \nabla \phi^{(2)} \cdot \nabla \opsi_t 
\end{aligned}
\end{equation}
with zero initial data. We test this equation with $\opsi_t$ and integrate over space and $(0,t)$ to arrive at the identity
\begin{equation} \label{id1_ContractivityKuzn}
\begin{aligned}
&\frac12\nLtwo{\sqrt{\alpha(t)}\,\opsi_t(t)}^2+\frac{c^2}{2}\nLtwo{\nabla \opsi(t)}^2+b \|\nabla \opsi_t\|_{L^2(0, t; L^2)}^2 \\
=&\,\begin{multlined}[t] -\frac{\kappa}{2} \int_0^t \intO \phi _{tt}^{(1)} \opsi^2_t \dxs +\kappa \int_0^t \intO \overline{\phi}_t\psi _{tt}^{(2)} \opsi_t \dxs\\
+  \sigma \int_0^t \int_{\Omega} (\nabla \overline{\phi} \cdot \nabla \psi^{(1)}_t+\nabla \phi^{(2)} \cdot \nabla \opsi_t) \opsi_t \dxs, \end{multlined}
\end{aligned}
\end{equation}
where $\alpha(t)=1-\kappa \phi^{(1)}_t(t)$. To treat the $\nabla \overline{\phi} \cdot \nabla \psi^{(1)}_t$ term on the right, we employ the following estimate: 
\begin{equation}
\begin{aligned}
\left|2\int_0^t \int_{\Omega} \nabla \overline{\phi} \cdot \nabla \psi^{(1)}_t \opsi_t \dxs \right| \leq&\, 2\|\nabla \psi^{(1)}_t\|_{L^2(0,t ;L^\infty)} \|\nabla \overline{\phi}\|_{L^\infty(0,t; L^2)} \|\opsi_t\|_{L^2(0,t; L^2)}.
\end{aligned}
\end{equation}
{Similarly to \eqref{self_m_1}--\eqref{self_m_2},} to estimate the $\nabla \phi^{(2)} \cdot \nabla \opsi_t$ term, we can rely on
\begin{equation} \label{est1_contractivity}
\begin{aligned}
\left|2\int_0^t \int_{\Omega} (\nabla \phi^{(2)} \cdot \nabla \opsi_t) \opsi_t \dxs \right| 
=&\, \left | \int_0^t \int_{\Omega} (- \D \phi^{(2)})  \opsi_t \opsi_t \dxs \right|\\
\leq&\,  \|\D \phi^{(2)}\|_{L^\infty(0,t; L^\infty)}\|\opsi_t\|_{L^2(0,t; L^2)}^2.
\end{aligned}
\end{equation}
Additionally, we can bound the two $\kappa$ terms on the right-hand side of \eqref{id1_ContractivityKuzn} as follows:
\begin{equation}
\begin{aligned}
&\left |-\frac12\kappa \int_0^t \intO \phi _{tt}^{(1)} \opsi^2_t \dxs +\kappa \int_0^t \intO \psi _{tt}^{(2)}\overline{\phi}_t \opsi_t \dxs \right| \\
\leq&\,\begin{multlined}[t] \frac12 |\kappa| \|\phi _{tt}^{(1)}\|_{L^\infty(0,t; L^\infty)}\|\opsi_t\|_{L^2(0,t; L^2)}^2 + |\kappa|\|\psi _{tt}^{(2)}\|_{L^2(0,t; L^\infty)} \|\overline{\phi}_t\|_{L^\infty(0,t; L^2)}\|\opsi_t\|_{L^2(0,t; L^2)}. \end{multlined}
\end{aligned}
\end{equation}
Employing the above estimates in \eqref{id1_ContractivityKuzn}, using Young's inequality, and then taking the supremum over 
$t \in (0, \tau)$ leads to
\begin{equation} \label{est1_ContractivityKuzn}
\begin{aligned}
&\frac12 \ulal\sup_{t\in(0,\tau)}\nLtwo{\opsi_t(t)}^2+\frac{c^2}{2}\sup_{t\in(0,\tau)}\nLtwo{\nabla \opsi(t)}^2+b \|\nabla \opsi_t\|^2_{L^2(0,\tau;L^2)} \\
\leq&\,\begin{multlined}[t] \frac12 |\kappa| \|\phi _{tt}^{(1)}\|_{L^\infty(L^\infty)}\|\opsi_t\|_{L^2(0,\tau;L^2)}^2+\frac12 |\kappa\|\psi _{tt}^{(2)}\|_{L^2(L^\infty)}|\|\opsi_t\|_{L^2(0,\tau;L^2)}^2  \\
+ \frac12|\kappa| \|\psi _{tt}^{(2)}\|_{L^2(L^\infty)}\|\overline{\phi}_t\|_{L^\infty(0,\tau;L^2)}^2
+\sigma\|\nabla \psi^{(1)}_t\|_{L^2(L^\infty)} \|\nabla \overline{\phi}\|_{L^\infty(0,\tau;L^2)}^2\\+\sigma\|\nabla \psi^{(1)}_t\|_{L^2(L^\infty)}\|\overline{\psi}_t\|_{L^2(0,\tau;L^2)}^2 +\sigma\|\D \phi^{(2)}\|_{L^\infty( L^\infty)}\|\opsi_t\|^2_{L^2(0,\tau;L^2)}. \end{multlined}
\end{aligned}
\end{equation}
Since the functions $\phi^{(1)}$ and $\phi^{(2)}$ belong to $\MK$ and functions $\psi^{(1)}$ and $\psi^{(2)}$ solve the linear problem, we have the uniform bounds 
\begin{equation}
\begin{aligned}
\|\phi _{tt}^{(1)}\|_{L^\infty(L^\infty)} \leq&\, \CHtwo \RK,  \\
\|\D \phi^{(2)}\|_{L^\infty( L^\infty)} \leq&\, \CHtwo \RK,\\
\|\psi _{tt}^{(2)}\|_{L^2(L^\infty)} \leq&\, \CHtwo \sqrt{T} C_{\textup{lin}, \textup{K}}(T,\RK) (\|\psi_0\|_{H^4}+\|\psi_1\|_{H^3}),  \\
\|\nabla \psi^{(1)}_t\|_{L^2(L^\infty)}  \leq& \, \CHtwo \sqrt{T}\|\nabla \psi^{(1)}_t\|_{L^\infty(H^2)} \\
 \leq&\, \begin{multlined}[t]\CHtwo \sqrt{T}\left((1+T+T^{3/2}C_{\textup{lin}, \textup{K}}(T,\RK))^2 +C_{\textup{lin}, \textup{K}}(T,\RK)^2\right)^{1/2}\\
 \times (\|\psi_0\|_{H^4}+\|\psi_1\|_{H^3}), \end{multlined}
\end{aligned}
\end{equation}
where we have estimated
\[
\begin{aligned}
\|\nabla \psi^{(1)}_t\|_{L^\infty(H^2)}^2
\leq&\,\|\nabla \psi^{(1)}_t\|_{L^\infty(L^2)}^2+\|\nabla \D \psi^{(1)}_t\|_{L^\infty(L^2)}^2
\\
\leq&\, (\nLtwo{\nabla \psi_1}+T\nLtwo{\nabla \psi_2}+T^{3/2}\nLtwoLtwo{\nabla \psi^{(1)}_{ttt}})^2
+\nLinfLtwo{\nabla \D\psi^{(1)}_t}^2.
\end{aligned}
\]
Therefore, applying Gronwall's inequality to \eqref{est1_ContractivityKuzn} results in the bound 
\begin{equation}
\begin{aligned}
\|\opsi\|^2_{\textup{E}} \leq&\, \begin{multlined}[t] C_1(1+T^{3/2}) \, \sqrt{T}\, (\|\psi_0\|_{H^4}+\|\psi_1\|_{H^3}) \\ \times \exp(C_2 (\RK+(1+T^{3/2})\sqrt{T}(\|\psi_0\|_{H^4}+\|\psi_1\|_{H^3}))\|\overline{\phi}\|^2_{\textup{E}}, \end{multlined}
\end{aligned}
\end{equation}
where the constants $C_1$, $C_2>0$ do not depend on $b$ nor $T$. By sufficiently reducing the final time so that 
\[ 
\begin{aligned}
\begin{multlined}[t]
C_1(1+T^{3/2}) \, \sqrt{T}\, (\|\psi_0\|_{H^4}+\|\psi_1\|_{H^3})\\ \times \exp(C_2 (\RK+(1+T^{3/2})\sqrt{T}(\|\psi_0\|_{H^4}+\|\psi_1\|_{H^3})) <1,
\end{multlined}
\end{aligned}
 \]
we can guarantee that
\[\|\psi^{(1)}-\psi^{(2)}\|_{\textup{E}} \leq r \|\phi^{(1)}-\phi^{(2)}\|_{\textup{E}} \ \text{ for some } r \in (0,1). \]
We note that it would be possible to impose smallness of data in the $H^4 \times H^3$ norm instead of reducing the final time, analogously to the proof of well-posedness for the Westervelt equation. \\
\indent We can argue similarly to the proof of Theorem~\ref{Thm:West_Wellposedness} that $(\MK, d_{\textup{E}})$ is a closed subset of a complete space with the metric induced by the energy norm $d_{\textup{E}}(x,y)=\|x-y\|_{\textup{E}}$. An application of Banach's Fixed-point theorem to $\mathcal{T}$ thus concludes the proof. 
\end{proof}	
\section{The inviscid limit of the Kuznetsov equation}  \label{Sec:LimitKuzn}
Having obtained the uniform bounds with respect to $b$ in the $\spaceK$ norm, we are now ready to derive the convergence rate for solutions of the Kuznetsov equation as $b \rightarrow 0^+$.
\begin{theorem}\label{Thm:Kuzn_WeakLimit}
	Under the assumptions of Theorem~\ref{Thm:WellpKuzn}, the family of solutions $\{\psi^{(b)}\}_{b>0}$ of the Kuznetsov equation converges in the energy norm	to a solution $\psi$ of the inviscid Kuznetsov equation at a linear rate. In other words,
	\begin{equation} \label{linrate_Kuzn}
	\|\psi^{(b)}-\psi\|_{\textup{E}}\lesssim b \ \mbox{ as } b\to0.
	\end{equation}
\end{theorem}
\begin{proof}
The proof follows along the lines of the proof of contractivity in the previous theorem. Let $b, b' \in (0, \overline{b})$. Let $\psi^{(b)}$ and $\psi^{(b')}$ be the solutions of the Kuznetsov equation with the sound diffusivity $b$ and $b'$, respectively, and with homogeneous Dirichlet data,  where we assume that initial conditions and final time are chosen according to the well-posedness result of Theorem~\ref{Thm:WellpKuzn}.\\
\indent The difference $\opsi=\psi^{(b)}-\psi^{(b')}$ then solves the equation
\begin{equation} \label{Kuzn_Cauchy_eq}
\begin{aligned}
&(1-\kappa \psi_t^{(b)} )\opsi_{tt}-c^2 \Delta \opsi-b \Delta \opsi_t\\
=&\, \begin{multlined}[t] \kappa \opsi_t \psi^{(b')} _{tt}+\sigma \nabla \opsi \cdot \nabla \psi^{(b)}_t+\sigma \nabla \psi^{(b')} \cdot \nabla \opsi_t+(b-b')\Delta \psi_t^{(b')}, \end{multlined}
\end{aligned}
\end{equation}
supplemented by zero initial conditions. Testing with $\opsi_t$ and integrating with respect to space and time yields
\begin{equation} \label{id1_CauchyKuzn}
\begin{aligned}
&\frac12\nLtwo{\sqrt{\alpha(t)}\opsi_t(t)}^2+\frac{c^2}{2}\nLtwo{\nabla \opsi(t)}^2+b \|\nabla \opsi_t\|_{L^2(0, t; L^2)}^2 \\
=&\,\begin{multlined}[t] -\frac{\kappa}{2} \int_0^t \intO \psi _{tt}^{(b)} \opsi^2_t \dxs +\kappa \int_0^t \intO \psi _{tt}^{(b')} \opsi^2_t \dxs\\
+  \sigma \int_0^t \int_{\Omega} (\nabla \opsi \cdot \nabla \psi^{(b)}_t+\nabla \psi^{(b')} \cdot \nabla \opsi_t) \opsi_t \dxs\\
+ (b-b')\int_0^t \intO  \D \psi_t^{(b')} \opsi_t \dxs, \end{multlined}
\end{aligned}
\end{equation}
where $\alpha(t)=1-\kappa \psi^{(b)}_t(t)$. We can rely on the following estimate: 
\begin{equation}
\begin{aligned}
\left|\int_0^t \int_{\Omega} \nabla \opsi \cdot \nabla \psi^{(b)}_t \opsi_t \dxs \right| \leq \|\nabla \psi^{(b)}_t\|_{L^\infty(0,t ;L^\infty)} \|\nabla \opsi\|_{L^2(0,t; L^2)} \|\opsi_t\|_{L^2(0,t; L^2)}.
\end{aligned}
\end{equation}
Similarly to \eqref{est1_contractivity}, we also have the bound
\begin{equation}
\begin{aligned}
\left|\int_0^t \int_{\Omega} (\nabla \psi^{(b')} \cdot \nabla \opsi_t) \opsi_t \dxs \right| 
=&\, \left |\frac12 \int_0^t \int_{\Omega} (- \D \psi^{(b')})  \opsi_t \opsi_t \dxs \right|\\
\leq&\, \frac12 \|\D \psi^{(b')}\|_{L^\infty(0,t; L^\infty)}\|\opsi_t\|_{L^2(0,t; L^2)}^2.
\end{aligned}
\end{equation}
The two $\kappa$ terms on the right-hand side of \eqref{id1_CauchyKuzn} can be estimated as follows:
\begin{equation}
\begin{aligned}
&\left |-\frac12\kappa \int_0^t \intO \psi _{tt}^{(b)} \opsi^2_t \dxs +\kappa \int_0^t \intO \psi _{tt}^{(b')} \opsi^2_t \dxs \right| \\
\leq&\, |\kappa| \left(\frac12 \|\psi _{tt}^{(b)}\|_{L^\infty(0,t; L^\infty)}+ \|\psi _{tt}^{(b')}\|_{L^\infty(0,t; L^\infty)}\right)\|\opsi_t\|_{L^2(0,t; L^2)}^2. 
\end{aligned}
\end{equation}
Thus, employing the above estimates in \eqref{id1_CauchyKuzn}, Young's inequality, and taking the supremum over $t \in (0, \tau)$ leads to
\begin{equation} \label{est1_CauchyKuzn}
\begin{aligned}
&\frac12 \ulal\sup_{t\in(0,\tau)}\nLtwo{\opsi_t(t)}^2+\frac{c^2}{2}\sup_{t\in(0,\tau)}\nLtwo{\nabla \opsi(t)}^2+b \|\nabla \opsi_t\|^2_{L^2(0,\tau; L^2)} \\
\leq&\,\begin{multlined}[t] |\kappa| \left(\frac12\|\psi _{tt}^{(b)}\|_{L^\infty(L^\infty)}+\|\psi _{tt}^{(b')}\|_{L^\infty( L^\infty)}\right)\|\opsi_t\|_{L^2(L^2)}^2 \\
+  \frac12|\sigma|\|\nabla \psi^{(b)}_t\|^2_{L^\infty(L^\infty)} \|\nabla \opsi\|_{L^2( L^2)}^2+ \frac12|\sigma|\|\opsi_t\|^2_{L^2( L^2)}\\+\frac12|\sigma|\|\D \psi^{(b')}\|_{L^\infty(L^\infty)}\|\opsi_t\|_{L^2( L^2)}^2 + \frac12 (b-b')^2 \|\D \psi_t^{(b')}\|_{L^2( L^2)}^2+\frac12\|\opsi_t\|_{L^2( L^2)}^2. \end{multlined}
\end{aligned}
\end{equation}
By virtue of Theorem~\ref{Thm:WellpKuzn}, we have the uniform bound
\begin{equation}
\begin{aligned}
\begin{multlined}[t]  \nLinfLinf{\psi _{tt}^{(b)}}+\nLinfLinf{\psi _{tt}^{(b')}}+\|\nabla \psi^{(b)}_t\|^2_{L^\infty(L^\infty)} \\+\nLinfLinf{\D \psi^{(b')}} \lesssim R_{\textup{K}} + \RK^2, \end{multlined}
\end{aligned}
\end{equation}
and $\nLtwoLtwo{\D \psi_t^{(b')}}^2 \lesssim T R_{\textup{K}}^2$. Therefore applying Gronwall's inequality results in
\begin{equation}
\|\opsi\|_{\textup{E}}=\|\psi^{(b)}-\psi^{(b')}\|_{\textup{E}} \lesssim |b-b'|.
\end{equation}
Finally, we note that the estimate remains valid for $b'=0$ and thus \eqref{linrate_Kuzn} holds, as claimed.
\end{proof}

\section*{Conclusion}
In this work, we have performed a convergence study of parabolic approximations to quasilinear wave equations of second order as the damping parameter $b$ tends to zero. Our research has been motivated by two classical nonlinear acoustics models, the Westervelt and Kuznetsov equations. As the analysis showed, the solutions of these equations converge in the energy norm with a linear rate to the solutions of the respective inviscid equations (with $b=0$), provided that the initial data are smooth and small and the final time short enough. For a short enough final time, the smallness condition on the initial data can be imposed in a lower-order space than the one needed to derive higher-energy bounds. In particular, the smallness condition on the data only arises from ensuring that the non-degeneracy assumption on $p$ and $\psi_t$ holds in the Westervelt and Kuznetsov equations, respectively.
\section*{Acknowledgment}
\noindent
The work of the first author was supported by the Austrian Science Fund {\sc fwf} under the grants P30054 and DOC 78.
\end{document}